\documentclass[11pt,reqno]{amsart}
\usepackage{amsmath, amssymb, amsthm}
\usepackage{url}
\usepackage{relsize}
\usepackage[breaklinks]{hyperref}
\usepackage{autonum}
\usepackage{cite}
\usepackage{csquotes}
\usepackage{color}

\theoremstyle{plain}
\newtheorem{thm}{Theorem}[section]
\theoremstyle{plain}
\newtheorem{lem}[thm]{Lemma}

\newtheorem{cor}[thm]{Corollary}

\theoremstyle{definition}
\newtheorem{defi}{Definition}[section]
\newtheorem{rem}{Remark}[section]

\numberwithin{equation}{section}

\begin{document}

\title{On \(\mathcal{B}^4\)-Almost Periodicity for a Class of Arithmetical Functions}

\date{}

\author{Kritika Aggarwal}
\address{Kritika Aggarwal\\ Department of Mathematics\\
Indraprastha Institute of Information Technology IIIT, Delhi\\
Okhla, Phase III, New Delhi-110020, India.} 
\email{kritikaa@iiitd.ac.in}

\author{Debika Banerjee}
\address{Debika Banerjee\\ Department of Mathematics\\
Indraprastha Institute of Information Technology IIIT, Delhi\\
Okhla, Phase III, New Delhi-110020, India.} 
\email{debika@iiitd.ac.in}

\author{Shubham Gupta}
\address{Shubham Gupta\\ Department of Mathematics\\
Indraprastha Institute of Information Technology IIIT, Delhi\\
Okhla, Phase III, New Delhi-110020, India.} 
\email{shubhamgupta2587@gmail.com, shubhamg@iiitd.ac.in}


\begin{abstract} 
In this paper, we establish the \(\mathcal{B}^4\)-almost periodicity in the sense of Besicovitch for a suitably normalized error term associated with a broad class of arithmetical functions introduced by Chandrasekharan and Narasimhan. This result significantly strengthens \(\mathcal{B}^2\)-almost periodicity previously investigated in the literature for related error terms. By deriving truncated Vorono\"{i}-type formulas, we demonstrate that the normalized error term lies in the Besicovitch space $B^4$. As a consequence, we deduce that the error term admits a limit probability distribution and establish an explicit formula for its mean fourth power moment. 
\end{abstract}

\maketitle

\section{Introduction}
Bohr  \cite{BO19251, BO19252, BO19253} introduced the theory of almost periodic functions in 1925, as a natural generalization of periodic functions. 
His theory was developed primarily for uniformly continuous functions, although this setting was not the most general one for which his methods and results hold true.  This motivated subsequent efforts to extend the theory to broader classes of functions. The first such extension was given by Stepanoff \cite{STE1926}, who removed the continuity assumption and characterized almost periodicity through mean values over intervals of fixed length rather than pointwise function values. A further and more comprehensive generalization was introduced by Besicovitch \cite{BES1926}, who defined almost periodic functions as limits of trigonometric polynomials with respect to a mean convergence that is weaker than uniform convergence.

For each \(q \geq 1\), Besicovitch introduced the space \(B^q\) of generalized almost periodic functions, which serves as an analogue of the Lebesgue space \(L^q\) for periodic functions. His primary objective was to enlarge the class of almost periodic functions while preserving essential Hilbert and Banach space properties, most notably the validity of the Riesz--Fischer theorem in the case \(q=2\). In recent years, Besicovitch spaces have played an important role in investigating the distributional properties of arithmetic error terms. In particular, establishing \(\mathcal{B}^2\)- or, more generally, 
\(\mathcal{B}^q\)-almost periodicity provides a powerful tool for proving the existence of limit distributions and obtaining higher-moment information. Since the present work is concerned with the stronger notion of \(\mathcal{B}^4\)-almost periodicity, we recall below the definitions of \(\mathcal{B}^q\)-almost periodic functions and the Besicovitch space \(B^q\), along with a special kind of space which we denote by $H$-space.

\begin{defi}
 Let $1 \leq q < \infty$. We say that a measurable function $f: [1, \infty) \rightarrow \mathbb{C}$ is $\mathcal{B}^q$-almost periodic if, for every $\varepsilon > 0$,  there exists a trigonometric polynomial
 \begin{align}\label{tri_poly}
     p_J(t) = \sum_{j = 1}^J c_je^{2 \pi i \alpha_jt}
 \end{align}
having $c_j \in \mathbb{C}$ and $\alpha_j \in \mathbb{R}$ such that
\begin{align}\label{def_Bq}
    ||f - p_J||_q := \left( \limsup_{X \rightarrow \infty} \dfrac{1}{X}\int_1^{X}|f(t) - p_J(t)|^q dt\right)^{1/q} < \varepsilon.
\end{align}   
For $q\geq 1$, the collection of all $\mathcal{B}^q$-almost periodic functions is called the Besicovitch space $B^q$ of almost periodic functions (see \cite{Besicovitch}). In fact, for $f(t)\in B^q$ we define
\begin{align}
    ||f(t)||_{q}=\left( \lim_{X \rightarrow \infty} \dfrac{1}{X}\int_1^{X}|f(t)|^q dt\right)^{1/q},
\end{align}
which is only a semi-norm but not a norm, since $||f(t)||_{q}=0$ does not imply that $f(t)\equiv 0$.
\end{defi}

\begin{defi}[Space H]
    A function $f(t)$ lies in space $H$ if there exists a trigonometric polynomial  $p_J(t)$ defined in \eqref{tri_poly} such that
    $$
    \lim_{J\rightarrow\infty}\limsup_{X \rightarrow \infty} \left( \dfrac{1}{X} \int_1^X \min\{1, |f(t) - p_J(t)| \} dt\right) =0.
    $$
\end{defi}
\begin{rem}\label{remark1}
    By the above definition, we may easily observe that $B^1\subset H$. For $1 \leq q < p$, we have $B^p \subset B^q$ and $||f(t)||_{p} \geq ||f(t)||_{q}$.
Hence, $B^q\subset H$ for all $q\geq 1$.
    \end{rem}

Following Besicovitch's work, F\o lner \cite{EFOL1946, EFOL1957} extended the theory of almost periodic functions to arbitrary infinite groups. The significance of almost periodicity lies in its close connection with general trigonometric series through Fourier expansions. Furthermore, \(\mathcal{B}^q\)-almost periodic functions admit a limit probability distribution, a concept that has found important applications in analytic number theory and can be traced back to the work of Heath-Brown \cite{Brown}. The notion of a limit probability distribution is defined as follows.

\begin{defi}[Limit probability distribution]\label{probabilitydistribution}
Let $\nu_{\Re}$ and $\nu_{\Im}$ be probability measures on the Borel sets of $\mathbb{R}$. We say that a function $f: \mathbb{R}_{\geq 0} \rightarrow \mathbb{C}$ has a distribution $\nu=\nu_{\Re}+i\nu_{\Im}$ if
$$
\lim_{X \rightarrow \infty} \left(\frac{1}{X}\int_0^X g(\Re(f(t)))dt\right)=
\int_{\mathbb{R}} g(t)d\nu_{\Re}(t),
$$
and the analogous statement holds for $\Im(f(t))$ with the probability measure $\nu_{\Im}$, for every bounded continuous complex-valued function $g$ on $\mathbb{R}$.
\end{defi}

Let $\alpha=(\alpha_1,\alpha_2)\in[0,1)^2$ be fixed, and define
$$N_{\alpha}(R)=\#\{n\in\mathbb{Z}^2\big||n-\alpha|\leq R\},$$
which counts the lattice points inside the disk of radius $R$ centered at $\alpha$. The problem of estimating the error term $N_{\alpha}(R)-\pi R^2$ is known as the
\enquote{shifted circle problem}. When $\alpha=0$, it reduces to the classical
Gauss circle problem. In 1916, Hardy~\cite{HAR1916} conjectured that for every
$\varepsilon>0$,
\begin{align}\label{circleproblem}
|N_0(R)-\pi R^2|=O\left(R^{\frac 12+\varepsilon}\right).
\end{align}
This conjecture remains open to date. Recently, Li and Yang \cite{LY2023} obtained the best result in this direction, claiming that $$|N_0(R)-\pi R^2|\ll_{\varepsilon} R^{\theta+\varepsilon},$$ where $\theta=0.628966\cdots$, improving Huxley's bound~\cite{Huxley2003} for the first time in over twenty years.

Define the suitably normalized error term
\begin{align}\label{generalerrorterm}
    F_{\alpha}(R)=\frac{N_{\alpha}(R)-\pi R^2}{\sqrt{R}},
\end{align}
assuming that $R$ is uniformly distributed on $[0,T]$ for some $T>0$, $T\rightarrow \infty$. Heath-Brown \cite{Brown} examined this error term for $\alpha=0$, proving that $F_0(R)$ belongs to the Besicovitch space $B^2$, and hence to the space $H$. He further showed that every function in $H$ whose Fourier coefficients satisfy certain additional conditions admits a limit probability distribution, implying that the same holds for $F_0(R)$. He also proved that the corresponding distribution $\nu_0(dx)$ is absolutely continuous with respect to the Lebesgue measure, with density $p_0(x)$ given by an entire function that decays faster than any polynomial as $|x|\to\infty$. Bleher, Cheng, Dyson, and Lebowitz \cite{BCDL} extended Heath-Brown's method to establish the existence of such a distribution for $F_\alpha(R)$ for all $\alpha\in[0,1)^2$. They further proved that it is absolutely continuous and that its density decays at infinity like $\exp(-\mathrm{const\cdot}\,|x|^4)$.

Bleher \cite{BL1992} in 1992 initiated the study of the normalized error term 
$$F_{\gamma}(R;\alpha)=\frac{N_{\gamma}(R;\alpha)-\text{Area}\hspace{1mm}\Omega_{\gamma}(R)}{\sqrt{R}},$$
where $N_{\gamma}(R;\alpha)=\big|\Omega_{\gamma}(R)\cap(\alpha+\mathbb{Z}^2)\big|$ for a fixed $\alpha\in\mathbb{R}^2$, and $\Omega_{\gamma}(R)$ is the domain enclosed by a simple closed smooth curve $\gamma$ in $\mathbb{R}^2$, dilated by a factor of $R$. Note that $F_{\gamma}(R;\alpha)$ is a generalization of \eqref{generalerrorterm} to convex ovals. He showed that $F_{\gamma}(R;\alpha)$ lies in the $H$-space via $\mathcal{B}^2$-almost periodicity. His breakthrough result established the existence of a limit probability distribution for every function in the $H$-space under the original hypotheses alone, thereby eliminating the need for the additional assumptions required in Heath-Brown's earlier work. As a result, Bleher showed the existence of distribution of $F_{\gamma}(R,\alpha)$ on the half-line $R>0$. He also established an explicit formula for the variance of this distribution, for almost all $\alpha$ including the origin (see \cite[Equation (5.2)]{BL1992}). In the following theorem, we state Bleher's result in its precise form.

\begin{thm}\cite[Theorem 4.1]{BL1992}\label{Blehertheorem} 
   Every $f(t) \in H$ has a limit distribution $\nu$ as defined in Definition \ref{probabilitydistribution}.
\end{thm}

In conjunction with Remark \ref{remark1}, this proves that every almost periodic function in the Besicovitch space $B^q$ for $q \geq 1$, admits a limit probability distribution. We will use Theorem~\ref{Blehertheorem} at the end of Section~\ref{mainresults} to show that the error term associated with a certain class of arithmetical functions admits a limit probability distribution.

Gath applied almost periodicity to prove limit probability distributions in a non-commutative setting. For the Heisenberg lattice point problem introduced by Garg, Nevo, and Taylor~\cite{RNT}, Gath~\cite{Gath0,Gath2,Gath1,Gath3} obtained the best-known error estimates in dimensions $3$ and $2q+1$ ($q>1$) and established limit probability distributions for the normalized error terms using their $\mathcal{B}^2$-almost periodic nature.

  Another consequence of almost periodicity is the existence of moments. If $f\in B^q$, then for all $1\leq \kappa\leq q$,
\begin{align}
    \lim_{X\rightarrow \infty} \frac{1}{X}\int_0^X |f(x)|^{\kappa}dx
\end{align}
exists. In 1999, Peter \cite{Peter} examined the truncated formula for the remainder in an asymptotic formula connected with a Dirichlet $L$-series to deduce that it is $\mathcal{B}^2$-almost periodic. Using this, he further deduced the existence of distribution and second power moment of the error term. Šleževičienė and Steuding \cite{Steuding&S} considered the twisted circle problem
\begin{align}
    \sum_{n\leq x}r(n)\exp\left(\frac{2\pi ink}{4l}\right), \ \ \ \ r(n)=\#\{(a,b)\in\mathbb{Z}^2|a^2+b^2=n\},
\end{align}
where $x>0$, $k,l$ are positive integers, $\gcd(k,4l)=1$. They derived its  truncated Vorono\"{i} formula and used it to show that the error term is $\mathcal{B}^2$-almost periodic and admits a limit probability distribution. Additionally, they also estimated the asymptotic mean-square formula and omega bounds for this sum.

\section{The Chandrasekharan--Narasimhan class of arithmetical functions}\label{theairfun}
In this paper, we work with a class of Dirichlet series satisfying a functional equation involving gamma factors. Let $\{\lambda_n\}$ and $\{\mu_n\}$ be two sequences of real numbers satisfying
\begin{align}
0<\lambda_1<\lambda_2<\cdots<\lambda_n\to\infty,\\
0<\mu_1<\mu_2<\cdots<\mu_n\to\infty.
\end{align}
Let $\{a(n)\}$ and $\{b(n)\}$ be two sequences of complex numbers, not
identically zero. For a complex variable $s=\sigma+it$, where $\sigma, t \in \mathbb{R}$, consider the Dirichlet series
\begin{align}\label{phishi}
\varphi(s)=\sum_{n=1}^{\infty}\frac{a(n)}{\lambda_n^{\,s}}
\quad \text{and} \quad
\psi(s)=\sum_{n=1}^{\infty}\frac{b(n)}{\mu_n^{\,s}}.    
\end{align}
Assume that both series converge in some right half–plane, and denote by
$\sigma_a$ and $\sigma_a^*$ their respective abscissae of absolute
convergence. Let $\Delta(s)$ be one of the gamma factors
\begin{align}
\Gamma(s), 
\qquad 
\Gamma\!\left(\tfrac12 s\right)\Gamma\!\left(\tfrac12(s-p)\right), 
\quad \text{or} \ \ \
\Gamma^2\!\left(\tfrac12(s+1)\right),\label{gammafactors}
\end{align}
where $p$ is a fixed integer. In these three cases, we take
$r$ to be an arbitrary real number, $p+1$, and $1$, respectively.
We say that $\varphi$ and $\psi$ satisfy the functional equation
\begin{align}\label{fnal}
\Delta(s)\varphi(s)=\Delta(r-s)\psi(r-s),
\end{align}
if there exists a meromorphic function $\chi(s)$ with the following properties:
\begin{enumerate}
\item[(i)] $\chi(s)=\Delta(s)\varphi(s)$ for $\sigma>\sigma_\varphi$, and
$\chi(s)=\Delta(r-s)\psi(r-s)$ for $\sigma<r-\sigma_\psi$;

\item[(ii)] $\displaystyle \lim_{|t|\to\infty}\chi(\sigma+it)=0$
uniformly for $\sigma$ in every finite vertical strip
$\sigma_1\le\sigma\le\sigma_2$;

\item[(iii)] all poles of\/ $\chi(s)$ lie in a compact subset of\/ $\mathbb{C}$.
\end{enumerate}
 This class of Dirichlet series was introduced by Chandrasekharan and Narasimhan \cite{CN0,CN2}, who obtained explicit representations of the associated error terms in terms of Bessel functions corresponding to the gamma factors listed in \eqref{gammafactors}. In case $\Delta(s)=\Gamma(s)$, Equation~\eqref{fnal} reduces to Hecke's functional equation.

Let $q\in\mathbb{Z}_{\geq 0}$ and $x>0$. Define
\[
A_q(x)=\frac{1}{\Gamma(q+1)}\sideset{}{'}\sum_{\lambda_n\leq x}a(n)(x-\lambda_n)^q,
\]
and
\[
Q_q(x)=\frac{1}{2\pi i}\int_{C_q}\frac{\Gamma(s)\varphi(s)}{\Gamma(s+q+1)}x^{s+q}\,ds,
\]
where $C_q$ is a contour enclosing all the poles of the integrand. Set
\begin{align}
D_q(x)=A_q(x)-Q_q(x).
\end{align}
Assume that for $\sigma>\sigma_a^*$,
\[
\sup_{0\leq h\leq 1}\left|\sum_{m^2\leq \mu_n\leq (m+h)^2}b(n)\mu_n^{\frac12-\sigma}\right|
=o(1)\qquad\text{as }m\to\infty.
\]
Then, Chandrasekharan and Narasimhan \cite{CN0,CN2} proved that for $q>2\sigma_a^*-r-\frac32$,
\begin{align}
D_q(x)=\sum\frac{b(n)}{\mu_n^{r+q}}I_q(\mu_nx),\label{D_qsum}
\end{align}
where $I_q$ is given, in the three cases of \eqref{gammafactors}, respectively by
\begin{align}
I_q(x)&=x^{\frac{r+q}{2}}J_{r+q}(2\sqrt{x}),\\
&\hspace{0.5cm} \text{or}\\
&=x^{\frac{p+q+1}{2}}2^{-q}\left\{\cos\left(\frac{(p+1)\pi}{2}\right)J_{p+q+1}(4\sqrt{x})\right.\\
&\left.\hspace{2.3cm}-\sin\left(\frac{(p+1)\pi}{2}\right)\left[Y_{p+q+1}(4\sqrt{x})+\frac{2(-1)^{p+q}}{\pi}K_{p+q+1}(4\sqrt{x})\right]\right\},\\
&\hspace{0.5cm}\text{or}\\
&=x^{\frac{q+1}{2}}2^{-q}\left\{Y_{q+1}(4\sqrt{x})+\frac{2(-1)^{q+1}}{\pi}K_{q+1}(4\sqrt{x})\right\}.
\end{align}
The series of Bessel functions in \eqref{D_qsum} converges uniformly on every interval in $(0,\infty)$ on which $D_q(x)$ is continuous. When $q=0$, it converges boundedly on every compact interval in $(0,\infty)$. We also refer the reader to the series of papers by Berndt \cite{Berndt12,Berndt3,Berndt4,Berndt56,Berndt7}, where various properties of the associated arithmetical functions are studied in detail.

In this paper, we establish the $\mathcal{B}^4$-almost periodicity for the error term associated with this class of arithmetical functions, satisfying some additional constraints. More specifically, we work with the gamma factors $\Gamma(s)$, $\Gamma^2\left(\frac 12(s+1)\right)$ and $\Gamma\!\left(\tfrac12 s\right)\Gamma\!\left(\tfrac12(s-p)\right)$ for $p=0$ or $p$ being a positive odd integer only. For this purpose, we first derive the truncated Vorono\"{i} formulas for the error term in each of the three cases separately, corresponding to the choice of gamma factor.

\subsection{Notations and Conventions} 
Throughout this paper, we assume that $\varphi=\psi$, so that $a(n)=b(n)$ and $\lambda_n=\mu_n=C_{\varphi}n$ for all $n\in\mathbb{N}$, where $C_{\varphi}$ is a positive constant depending on the Dirichlet series $\varphi$. We further assume that
$$
b(n)\ll n^{\varepsilon}.
$$
Let $\sigma_a$ denote the abscissa of absolute convergence of $\varphi$, and let $r$ be as defined in \eqref{fnal}. We  define
\begin{align}\label{defidelta_0}
\delta_0=2\sigma_a-r-\frac12.
\end{align}
Throughout the paper, we assume that
\begin{align}\label{defidelta_0con}
0\leq \delta_0\leq \frac12.
\end{align}
We also define the summatory function
\begin{align}
A_{0}(x)=\sideset{}{'}\sum_{n\leq x} b(n), \label{sum0}
\end{align}
where the prime indicates that, if $x\in\mathbb{N}$, the term $b(x)$ is counted with weight $\tfrac12$. 

Unless otherwise specified, $\varepsilon>0$ denotes an arbitrarily small positive constant, not necessarily the same at each occurrence. For any real number $l$, the notation $\int_{(l)}$ denotes the line integral $\int_{l-i\infty}^{l+i\infty}$. Finally, $\lfloor\cdot\rfloor$ denotes the greatest integer function.

The paper is organized as follows. In Section \ref{mainresults}, we describe the main results of the paper. In Section~\ref{prelimm}, we recall several preliminary results that serve as the key ingredients for proving our main results. Section~\ref{Vorr} is devoted to the proofs of Theorems \ref{truncatedformula}, \ref{truncatedformula2}, and~\ref{truncatedformula3}. In Section \ref{mainproof}, we use these results to prove Theorem~\ref{th1} and Corollary \ref{cor:fourth_moment}. Finally, the concluding section presents illustrative examples of arithmetical functions that belong to the class of Dirichlet series under consideration.

\section{Main Results}\label{mainresults}
Theorems \ref{truncatedformula}--\ref{truncatedformula3} provide truncated Voronoï formulas for the error term associated with the summatory function $A_0(x)$ defined in \eqref{sum0}, corresponding to the three distinct choices of $\Delta(s)$ and $r$ in \eqref{fnal}. Throughout, let $x \geq 1$, and let $D_{0}(x)$ denote the error term in the asymptotic formula for $A_0(x)$. We now state our main results.

\begin{thm}\label{truncatedformula}
Let $\varphi$ be as defined in \eqref{phishi} 
 satisfying the functional equation \eqref{fnal} with $\Delta(s) = \Gamma(s)$ and $r$ being any fixed positive real.
 Then, for $N\in\mathbb{N}$  we have
    \begin{align}
        D_{0}(x)=\frac{1}{\sqrt{\pi}}C_{\varphi}^{-\frac{r}{2}-\frac14} &x^{\frac{r}{2}-\frac14}\sum_{n=1}^{N}\frac{b(n)}{ n^{\frac{r}{2}+\frac14}}  \cos\left(
2\sqrt{C_{\varphi}nx}-\frac{r \pi}{2}-\frac{\pi}{4}
\right)\\
&~+\begin{cases}
        O\left(x^{\frac{\delta_0}{2}+\frac{r}{2}-\frac14+\varepsilon} N^{-\frac12}  +  x^{-\frac{\delta_0}{2}+\frac r2-\frac 14-\varepsilon}\right) &\text{if }\ \delta_0<\frac 12\\
        O\left(x^{\frac{\delta_0}{2}+\frac{r}{2}-\frac14+\varepsilon} N^{-\frac12}  +  x^{-\frac{\delta_0}{2}+\frac r2-\frac 14+\varepsilon}N^{\varepsilon}\right) &\text{if }\ \delta_0=\frac 12.
    \end{cases}\label{truncated1}
    \end{align}
\end{thm}

\begin{thm}\label{truncatedformula2}
Let $\varphi$ be as defined in \eqref{phishi} satisfying the functional equation \eqref{fnal} with $\Delta(s) = \Gamma(\frac{s}{2})\Gamma\left(\frac{s-p}{2}\right)$ and $r=p+1$. Then, for $N\in\mathbb{N}$, we have
   \begin{enumerate}
       \item \label{truncated2.1} if $p=0$ and $\sigma_a=1$ (or equivalently, $\delta_0=\frac 12$), we have
       \begin{align}
        D_{0}(x)=\frac{1}{\sqrt{2\pi}}C_{\varphi}^{-\frac34} x^{\frac14}&\sum_{n=1}^{N}\frac{b(n)}{ n^{\frac34}}  \cos\left(
4\sqrt{C_{\varphi}nx}-\frac{\pi}{4}
\right)+
O\left(x^{\frac 12+\varepsilon} N^{-\frac12}  +  x^{\varepsilon}N^{\varepsilon}\right).
    \end{align}
    \item \label{truncated2.2} if $p$ is a positive odd integer, we have
 \begin{align}
        D_{0}(x)=\frac{1}{\sqrt{2\pi}}&C_{\varphi}^{-\frac{p}{2}-\frac34}x^{\frac{p}{2}+\frac14}\sum_{n=1}^{N}\frac{b(n)}{ n^{\frac{p}{2}+\frac34}}  \cos\left(
4\sqrt{C_{\varphi}nx}-\frac{\pi}{4}
\right)\\
&\quad +
\begin{cases}
    O\left(x^{\frac{\delta_0}{2}+\frac{p}{2}+\frac14+\varepsilon} N^{-\frac12}  +  x^{-\frac{\delta_0}{2}+\frac p2+\frac 14-\varepsilon}\right) &\text{if }\ \delta_0<\frac 12\\
O\left(x^{\frac{\delta_0}{2}+\frac{p}{2}+\frac14+\varepsilon} N^{-\frac12}  +  x^{-\frac{\delta_0}{2}+\frac p2+\frac 14+\varepsilon}N^{\varepsilon}\right) &\text{if }\ \delta_0=\frac 12.\label{truncated2}
\end{cases}
    \end{align}

   \end{enumerate}
\end{thm}

\begin{thm}\label{truncatedformula3}
Let $\varphi$ be as defined in \eqref{phishi} with abscissa of absolute convergence $\sigma_a=1$, satisfying the functional equation \eqref{fnal} with $\Delta(s) = \Gamma^2(\frac{s + 1}{2})$  and $r=1$ (or equivalently, $\delta_0=\frac 12)$. Then, for $N\in\mathbb{N}$, we have
    \begin{align}
         D_{0}(x)=\frac{1}{\sqrt{2\pi}}C_{\varphi}^{-\frac34} x^{\frac14}\sum_{n=1}^{N}\frac{b(n)}{ n^{\frac34}} \cos&\left(
4\sqrt{C_{\varphi}nx}-\frac{\pi}{4}
\right)+O\left(x^{\frac 12+\varepsilon} N^{-\frac12}  +  x^{\varepsilon}N^{\varepsilon}\right).\label{truncated2}
    \end{align}
\end{thm}

We next prove that the error term $D_{0}(x)$, suitably normalized, belongs to the Besicovitch space $B^{4}$, a considerably stronger property than $\mathcal{B}^{2}$-almost periodicity. More precisely, we prove the following theorem.

\begin{thm}\label{th1}
Let $r$, $\sigma_a$, and $\delta_0$ satisfy the hypotheses of one of the preceding theorems. Furthermore, suppose that $r-\delta_0\geq\frac{1}{2}$ whenever $\delta_0\neq 0$ with strict inequality when $\delta_0=0$. Then the function $P:[1,\infty)\to\mathbb{C}$ defined by
\begin{align}
P(t)=t^{-r+\frac12}D_{0}(t^2)\label{P(t)defn}
\end{align}
is $\mathcal{B}^4$-almost periodic, and thus admits a limit probability distribution as defined in Definition \ref{probabilitydistribution}. 
\end{thm}

\begin{rem}
    In view of Remark \ref{remark1}, $P(t)$ defined in \eqref{P(t)defn} lies in the Besicovitch space $B^q$ for $q=1,2,3,4$, and hence $P(t)\in H$. Therefore, by Theorem \ref{Blehertheorem}, it possesses a limit probability distribution.
\end{rem}

\begin{cor}[Fourth Power Moment Formula]
\label{cor:fourth_moment}
 Under the assumptions of Theorem \ref{th1}, the fourth power moment of the normalized error term $P(t) = t^{-r+\frac 12}D_{0}(t^2)$ exists and is explicitly given by
\[
\lim_{X \to \infty} \frac{1}{X} \int_1^X |P(t)|^4 \, dt = \mathcal{M}_4,
\]
where
\begin{align}
\mathcal{M}_4 &= \frac{3}{8\pi^2} C_{\varphi}^{-2r-1} \sum_{\substack{n,m,k,l \ge 1 \\ \sqrt{n}+\sqrt{m}=\sqrt{k}+\sqrt{l}}} \frac{b(n)\overline{b(m)}b(k)\overline{b(l)}}{(nmkl)^{\frac{r}{2}+\frac{1}{4}}} \\
&- \frac{\sin(r\pi)}{2\pi^2} C_{\varphi}^{-2r-1} \sum_{\substack{n,m,k,l \ge 1 \\ \sqrt{n}+\sqrt{m}+\sqrt{k}=\sqrt{l}}} \frac{b(n)\overline{b(m)}b(k)\overline{b(l)}}{(nmkl)^{\frac{r}{2}+\frac{1}{4}}}.
\end{align}
\end{cor}

\section{Preliminaries}\label{prelimm}

We begin by collecting several auxiliary lemmas that will be used throughout the proofs in the subsequent sections.

\begin{lem}\cite[\S 21]{Rademacher}\label{Gamma}
Let $\delta<\pi$ be a fixed positive number. Then in any vertical strip $a\leq\sigma\leq b$ with $|arg\hspace{1.5mm}s|\leq \pi-\delta$ and $|t|\geq 1$, 
\begin{align}
					\Gamma(s)= \sqrt{2\pi}\ |t|^{\sigma-\frac{1}{2}}e^{-\frac{1}{2}\pi |t|}
        e^{i(t\ln |t|-t+ \operatorname{sgn}(t) \frac{\pi}{2} (\sigma-\frac12) )}
     \left( 1+O\left(\frac{1}{|t|}\right)\right).\label{Gammaasym}
				\end{align}
			\end{lem}


\begin{lem}\cite[Lemma 4.3]{Titchmarsh}\label{expintegrallemma}
Let $F(x)$ and $G(x)$ be two real-valued functions in an interval $[a,b]$, where $G(x)$ is continuous and $F(x)$ is continuously differentiable. Suppose that $G(x)/F'(x)$ is monotonic and $|F'(x)/G(x)|\geq m>0$. Then
$$\left|\int_a^bG(x)e^{iF(x)}dx\right|\leq \frac 4m.$$
\end{lem}

\begin{lem}\cite[Lemma 3.12]{Titchmarsh}\label{Perron}
Let $F(s)=\sum_{n=1}^{\infty}a(n)n^{-s}$ be a Dirichlet series with abscissa
of absolute convergence $\sigma_a$. For $c>\max(0,\sigma_a)$,
$T\ge1$, and $x=N+\tfrac12$ with $N\in\mathbb{N}$, we have
\begin{align}
\sideset{}{'}\sum_{n\le x} a(n)
=
\frac{1}{2\pi i}\int_{c-iT}^{c+iT}
F(s)\frac{x^s}{s}\,ds
+
O\!\left(\frac{x^{c}B(x)\log x}{T}\right)
+
O\!\left(\frac{x^{c}F(c)}{T}\right),
\end{align}
where
\[
B(x)=\max_{x/2\le n\le 3x/2}|a(n)|.
\]
\end{lem}

\begin{lem}\cite[\S 1.4]{Jutila}\label{jutilabesselrepresentation}
    Let $a$ be a nonnegative integer, $\sigma_1\geq -a/2$, $\sigma_2<-a$, $T>0$, and let $C_a$ be the contour joining the points $\sigma_1-i\infty$, $\sigma_1-iT$, $\sigma_2-iT$, $\sigma_2+iT$, $\sigma_1+iT$, and $\sigma_1+i\infty$ by straight lines. Let $X>0$ and $k$ be a positive integer. Then
    \begin{align}
        I_1:=\frac{1}{2\pi i}\int_{C_a}\Gamma^2(1-s)X^s(s(s+1)\dots (s+a))^{-1}ds=2(-1)^{a+1}X^{(1-a)/2}K_{a+1}(2X^{1/2}),\label{I1int}
    \end{align}
    and
    \begin{align}
        I_2:=\frac{1}{2\pi i}\int_{C_a}\Gamma^2(1-s)\cos (\pi s)X^s(s(s+1)\dots (s+a))^{-1}ds=\pi X^{(1-a)/2}Y_{a+1}(2X^{1/2}).\label{I2int}
    \end{align}
    Here, $Y_{\nu}(z)$ and $K_{\nu}(z)$ represent the Bessel functions of nonnegative integral order $\nu$.
\end{lem}

We shall require the following two lemmas of Tsang in the proof of Theorem~\ref{th1}.

\begin{lem}\cite[Lemma 3]{KT}
\label{lem1}
    If $n,m,k,l$ are natural numbers such that $\sqrt{n}+\sqrt{m}-\sqrt{k}-\sqrt{l}\neq 0$ or $\sqrt{n}+\sqrt{m}+\sqrt{k}-\sqrt{l}\neq 0$, then respectively,
    \begin{align}
        |\sqrt{n}+\sqrt{m}-\sqrt{k}-\sqrt{l}|\hspace{3mm} \text{or}\hspace{3mm} |\sqrt{n}+\sqrt{m}+\sqrt{k}-\sqrt{l}|\gg \max (n,m,k,l)^{-\frac72}.
    \end{align}
\end{lem}

\begin{lem} \cite[Lemma 4]{KT} \label{lem2}
    Let $\alpha$, $\beta$, and $\delta$ be real numbers such that $\alpha \neq 0$ and $0 < \delta < \frac{1}{2}$.
    Then 
    $$\# \{K<k<2K:||\beta+\alpha\sqrt{k}||<\delta\}\ll K\delta+|\alpha|^{\frac13}K^{\frac12}+|\alpha|^{-\frac12}K^{\frac34},$$
    where the implied constant is absolute. Here, the notation $||\cdot||$ denotes distance to the nearest integer.
\end{lem}


\section{Proof of Theorem \ref{truncatedformula}}\label{Vorr}
 Let $s=\sigma+it\in\mathbb{C}$. Since $\sigma_a=\frac{\delta_0}{2}+\frac{r}{2}+\frac 14>0$, therefore $\sigma_a>0$. Fix $c=\sigma_a+\varepsilon$ and
\begin{align}\label{Tdefn}
    T=\sqrt{C_{\varphi}x\left(N+\frac 12\right)}.
\end{align}
 Since $b(n)\ll n^{\varepsilon}$ for all $n\in\mathbb{N}$, it follows from Lemma~\ref{Perron} that
\begin{align}
     \sideset{}'\sum_{n\leq  x}  b (n)=&\frac{1}{2\pi i} \int_{c-iT}^{c+iT}   \frac{\varphi(s)x^{s}}{s  } ds+ E_0(x,T),\label{rsi}
 \end{align}
where  the error  $E_0(x,T)$ satisfies
 \begin{align} 
     E_0(x,T)\ll  \frac{x^{c+\varepsilon}\log x}{T }\ll \frac{x^{\frac{\delta_0}{2}+\frac{r}{2}+\frac14+\varepsilon} }{(Nx)^{  \frac12 }}\ll x^{\frac{\delta_0}{2}+\frac{r}{2}-\frac14+\varepsilon} N^{-\frac12} \label{erho new}.
 \end{align}
 To compute the integral on the right-hand side of \eqref{rsi}, we consider a positively oriented rectangular contour $\mathcal{S}$ with vertices $c-iT$, $-a_{\varepsilon}-iT$, $-a_{\varepsilon}+iT$ and $c+iT$, where $a_{\varepsilon}=\sigma_{a}-r+\varepsilon$ is a fixed real. Since the residues from the integral contribute towards the main term, we obtain that the error term is
 \begin{align}
     D_{0}(x)=\frac{1}{2\pi i} \left(\int_{-a_{\varepsilon}-iT}^{-a_{\varepsilon}+iT}+\int_{-a_{\varepsilon}+ iT}^{c+ iT}+\int_{c-iT}^{-a_{\varepsilon}-iT}\right) 
    \frac{\varphi(s)x^{s}}{s  } ds +E_0(x,T).\label{t:1}
 \end{align}
Employing \eqref{fnal} with $\Delta(s)=\Gamma(s)$, the Phragm\'{e}n-Lindel\"{o}f theorem \cite[pp. 66, \S 33]{Rademacher} gives
\[
\varphi(s)
\ll
\left(\varphi(c) + \varphi(a_{\varepsilon}+r)\right)\, t^{\frac{(2a_{\varepsilon}+r)(c-\sigma)}{c+a_{\varepsilon}}},
\]
for $t\gg 1$ and $-a_{\varepsilon}\leq\sigma\leq c$. Using this bound in \eqref{t:1}, we obtain that
\begin{align}
\int_{-a_{\varepsilon}\pm iT}^{c\pm iT}     \frac{\varphi(s)x^{s}}{s } ds&\ll \left(\varphi(c) + \psi(a_{\varepsilon}+r)\right)~ T^{\frac{(2a_{\varepsilon}+r)c}{c+a_{\varepsilon}}-1}\max\left\{\frac{x^{c}}{T^\frac{c(2a_{\varepsilon}+r)}{c+a_{\varepsilon}}},\frac{x^{-a_{\varepsilon}}}{T^{\frac{-a_{\varepsilon}(2a_{\varepsilon}+r)}{c+a_{\varepsilon}}}}\right\} \nonumber \\
&\ll \frac{x^{c}}{T}+x^{-a_{\varepsilon}}T^{2a_{\varepsilon}+r-1}=:E_1(x,T).\nonumber 
\end{align}
Then using \eqref{Tdefn} and $a_{\varepsilon}=\frac{\delta_0}{2}-\frac{r}{2}+\frac14+\varepsilon$, 
\begin{align}\label{erre1}
E_1(x,T)\ll x^{\frac{\delta_0}{2}+\frac{r}{2}-\frac14+\varepsilon} N^{-\frac12}  +  x^{\frac{r}{2}-\frac12+\varepsilon} N^{\frac12(\delta_0-\frac12)+\varepsilon}. 
\end{align}
Combining \eqref{rsi}, \eqref{erho new} and \eqref{erre1}, and using \eqref{fnal} with $\Delta(s)=\Gamma(s)$, we deduce that
\begin{align}
  D_{0}(x) &:= \frac{1}{2\pi i}  \int_{-a_{\varepsilon}-iT}^{-a_{\varepsilon}+iT} 
    \frac{\Gamma(r-s)\varphi(r-s)x^{s}}{\Gamma(s)s } ds+E_1(x,T)+E_0(x,T) \nonumber\\
    &=  \sum_{n=1}^{\infty}\frac{b(n)}{ (C_{\varphi}n)^{r}} \underbrace{ \frac{1}{2\pi i}  \int_{-a_{\varepsilon}-iT}^{-a_{\varepsilon}+iT} 
    \frac{\Gamma(r-s)(C_{\varphi}nx)^{s}}{\Gamma(s)s } ds}_{=:~I_{a_{\varepsilon}}(n)}+E_1(x,T)+E_0(x,T),\label{t:2}
 \end{align}
where in the last step we invert the order of summation and integration by absolute convergence on the
right-hand side.
 Next, we split the above sum in two parts and rewrite it as
 \begin{align}
   D_{0}(x)=C_{\varphi}^{-r}\left(  \sum_{n=1}^{N}\frac{b(n)}{ n^{r}} I_{a_{\varepsilon}}(n)+  \underbrace{\sum_{n=N+1}^{\infty}\frac{b(n)}{ n^{r}} I_{a_{\varepsilon}} (n)}_{=:E_2(x,T)}\right)+E_1(x,T)+E_0(x,T).\label{imst}
\end{align}
Let us first estimate the second sum. For $n\geq N+1$, we write the integral $I_{a_{\varepsilon}}(n)$ in the form
\begin{align}
(C_{\varphi}nx)^{a_{\varepsilon}}I_{a_{\varepsilon}}(n)&= \frac{1}{2\pi }\left( \int_{-d}^{d}+\int_{d}^T+\int_{-T}^{-d}\right)   f(-a_{\varepsilon}+it)\frac{(C_{\varphi}nx)^{it}}{(-{a_{\varepsilon}}+it) } dt\\
&=: I_{a_{\varepsilon}}^{(1)}(n)+I_{a_{\varepsilon}}^{(2)}(n)+I_{a_{\varepsilon}}^{(3)}(n),\label{I_a branching}
\end{align}
where $d=\max\{1,|a_{\varepsilon}|\}$ and $f(s)=\frac{\Gamma(r-s)}{\Gamma(s)}$.
It can be easily seen that 
for the vertical segment $|t|\leq d$, we have
\begin{align}
    I_{a_{\varepsilon}}^{(1)}(n)\ll 1. \label{I_a1bound}
\end{align}
Next, we estimate $I_{a_{\varepsilon}}^{(2)}(n)$. By Lemma \ref{Gamma}, we have
\begin{align}
  I_{a_{\varepsilon}}^{(2)}(n) &=\frac{e^{-\frac{i\pi r}{2}}}{2\pi}\int_{d}^T e^{2it (-\log t +1+\frac12 \log (C_{\varphi}nx))}   \{ t^{r-1+2a_{\varepsilon}}+O(t^{r-2+2a_{\varepsilon}})\}dt\nonumber\\
  &=\frac{e^{-\frac{i\pi r}{2}}}{2\pi}\int_{d}^T e^{2it (-\log t +1+\frac12 \log (C_{\varphi}nx))}   \left\{ t^{\delta_0-\frac12+\varepsilon}+O\left(t^{\delta_0-\frac 32+\varepsilon}\right)\right\}dt,\label{I_a^2exp}
\end{align}
where we have used the fact that
\begin{align}
\frac{1}{-a_{\varepsilon}+it}=-\frac{i}{t}+O\left(\frac{1}{t^2}\right).    
\end{align}
Since $\delta_0\leq \frac12$, we  use integration by parts in \eqref{I_a^2exp} to obtain
\begin{align}
   I_{a_{\varepsilon}}^{(2)}(n) =&\left(t^{\delta_0-\frac12+\varepsilon} \frac{e^{-\frac{i\pi r}{2}}}{2\pi}\int_{d}^t e^{2iu (-\log u +1+\frac12 \log C_{\varphi}nx)} du\right)\Bigg|_{d}^{T}\\
   &-\frac{e^{-\frac{i\pi r}{2}}}{2\pi}\left(\delta_0-\frac 12+\varepsilon\right)\left(\int_d^tu^{\delta_0-\frac 32+\varepsilon}\left(\int_d^ue^{2iw\left(-\log w+1+\frac 12\log (C_{\varphi}nx)\right)}dw\right)du\right)\Bigg|_d^T \\
   &+O\left(\int_d^T t^{\delta_0-\frac32+\varepsilon}dt\right)\nonumber.
   \end{align}
Applying Lemma \ref{expintegrallemma} to the inner integral in the second term on the right-hand side of the  above equation  with 
$$F(w)=2w\left(-\log w+1+\frac 12\log (C_{\varphi}nx)\right) \ \ \text{and} \ \ G(w)\equiv 1,$$
since $\frac{G(w)}{F'(w)}$ is monotonically increasing for $d\leq w\leq t\leq T$, we get that
\begin{align}
    \left|\frac{F'(w)}{G(w)}\right|=\log\left(\frac{C_{\varphi}nx}{w^2}\right)\geq \log\left(\frac{n}{N+\frac 12}\right).
\end{align}
This gives
   \begin{align}
   I_{a_{\varepsilon}}^{(2)}(n)&=  \left(t^{\delta_0-\frac12+\varepsilon} \frac{e^{-\frac{i\pi r}{2}}}{2\pi}\int_{d}^t e^{2iu (-\log u +1+\frac12 \log C_{\varphi}nx)} du\right)\Bigg|_{d}^{T}\\
   &\hspace{4cm}+O\left(\left(1+\frac{1}{\log\left(\frac{n}{N+\frac12}\right)}\right) \int_{d}^T t^{\delta_0-\frac32+\varepsilon}dt\right)\nonumber.
\end{align}
Again, applying Lemma \ref{expintegrallemma} to the integral in the main term above, we end up getting
\begin{align}
   I_{a_{\varepsilon}}^{(2)}(n)\ll
   \begin{cases}
       1+\frac{1}{\log\left(\frac{n}{N+\frac 12}\right)} &\text{if }\ \delta_0<\frac 12\\
       T^{\varepsilon}\left(1+\frac{1}{\log\left(\frac{n}{N+\frac 12}\right)}\right) &\text{if } \ \delta_0=\frac 12.\label{i2}
       \end{cases}
\end{align}
Same estimate holds for $I_{a_{\varepsilon}}^{(3)}(n)$ as well. The bounds in \eqref{I_a1bound} and \eqref{i2} together with \eqref{I_a branching} give
\begin{align}
\sum_{n=N+1}^{\infty}\frac{b(n)}{ n^{r}} I_{a_{\varepsilon}}(n) &
\ll \begin{cases}
x^{-{a_{\varepsilon}}}\left(1+  \sum_{n=N+1}^{\infty} \frac{|b(n)|}{n^{r+a_{\varepsilon}}}\left(1+O\left(\frac{1}{\log n}\right)\right)\right)&\text{if }\ \delta_0<\frac 12\\
    x^{-{a_{\varepsilon}}}\left(1+  T^{\varepsilon}\sum_{n=N+1}^{\infty} \frac{|b(n)|}{n^{r+a_{\varepsilon}}}\left(1+O\left(\frac{1}{\log n}\right)\right)\right) &\text{if }\ \delta_0=\frac 12.
    \end{cases}\nonumber
    \end{align}
    By simple computation, we obtain
    \begin{align}
\sum_{n=N+1}^{\infty}\frac{b(n)}{ n^{r}} I_{a_{\varepsilon}}(n)\ll \begin{cases}
    x^{-{a_{\varepsilon}}}&\text{if }\ \delta_0<\frac 12\\
    x^{-{a_{\varepsilon}}}T^{\varepsilon} &\text{if }\ \delta_0=\frac 12.
\end{cases}
\end{align}
Upon substituting the value of $a_{\varepsilon}$ and $T$ from \eqref{Tdefn}, we deduce that
\begin{align} \label{bound0}
E_2(x, T)\ll \begin{cases}
    x^{-\frac{\delta_0}{2}+\frac r2-\frac 14-\varepsilon}&\text{if }\ \delta_0<\frac 12\\
    x^{-\frac{\delta_0}{2}+\frac r2-\frac 14+\varepsilon}N^{\varepsilon}&\text{if }\ \delta_0=\frac 12.
\end{cases}
\end{align}

Now, we compute the first sum in the right-hand side of \eqref{imst}. We consider a contour $\mathcal{C}$ with counterclockwise orientation with vertices  $-a_{\varepsilon}-iT$, $-a_{\varepsilon}+iT$, $-a_{\varepsilon}+2\varepsilon+iT$  and $-a_{\varepsilon}+2\varepsilon-iT$. Since $\frac{\Gamma(r-s)}{s\Gamma(s)}$ is holomorphic in this region, no pole is encountered. Hence, 
\begin{align}
 I_{a_{\varepsilon}}(n)&=  \frac{1}{2\pi i}  \int_{-a_{\varepsilon}+2\varepsilon-iT}^{-a_{\varepsilon}+2\varepsilon+iT} 
    \frac{\Gamma(r-w)(C_{\varphi}nx)^{w}}{ w\Gamma(w) } dw+O\left(T^{r-1}\int_{-a_{\varepsilon}}^{-a_{\varepsilon}+2\varepsilon}\left(\frac{C_{\varphi}nx}{T^2}\right)^{\sigma}d\sigma\right) \nonumber \\
    &=\frac{1}{2\pi i}  \int_{-a_{\varepsilon}+2\varepsilon-iT}^{-a_{\varepsilon}+2\varepsilon+iT} 
    \frac{\Gamma(r-w)(C_{\varphi}nx)^{w}}{ \Gamma(w+1) } dw\\
    &\hspace{4cm}+O\left(T^{r-1}\max\left\{\left(\frac{C_{\varphi}nx}{T^2}\right)^{-a_{\varepsilon}},\left(\frac{C_{\varphi}nx}{T^2}\right)^{-a_{\varepsilon}+2\varepsilon}\right\}\right)\nonumber \\
     &=\frac{1}{2\pi i}  \int_{-a_{\varepsilon}+2\varepsilon-iT}^{-a_{\varepsilon}+2\varepsilon+iT} 
    \frac{\Gamma(r-w)(C_{\varphi}nx)^{w}}{ \Gamma(w+1) } dw+O\left(T^{r-1} \left(\frac{n}{(N+\frac12)}\right)^{-a_{\varepsilon}}\right),  \label{Ian}
   \end{align}
where the last step follows by the fact that $n\leq N$. Substituting \eqref{bound0} and \eqref{Ian} in \eqref{imst}, we get
 \begin{align}\label{MEN}
    D_{0}(x)&=C_{\varphi}^{-r}\left(  \sum_{n=1}^{N}\frac{b(n)}{ n^{r}} \underbrace{\frac{1}{2\pi i}  \int_{-a_{\varepsilon}+2\varepsilon-iT}^{-a_{\varepsilon}+2\varepsilon+iT} 
    \frac{\Gamma(r-w)(C_{\varphi}nx)^{w}}{\Gamma(w+1) } dw}_{L_{a_{\varepsilon}}(n)}\right)\\
    &\hspace{3cm}+\begin{cases}
        O\left(x^{\frac{\delta_0}{2}+\frac{r}{2}-\frac14+\varepsilon} N^{-\frac12}  +  x^{-\frac{\delta_0}{2}+\frac r2-\frac 14-\varepsilon}\right)&\text{if }\ \delta_0<\frac 12\\
        O\left(x^{\frac{\delta_0}{2}+\frac{r}{2}-\frac14+\varepsilon} N^{-\frac12}  +  x^{-\frac{\delta_0}{2}+\frac r2-\frac 14+\varepsilon}N^{\varepsilon}\right)&\text{if }\ \delta_0=\frac 12.
    \end{cases}
 \end{align}
 It remains to estimate 
 \begin{align}\label{J_a} 
 L_{a_{\varepsilon}}(n)&= \frac{1}{2\pi i}  \int_{(-a_{\varepsilon}+2\varepsilon)} \frac{\Gamma(r-w)(C_{\varphi}nx)^{w}}{ \Gamma(w+1) } dw-\frac{1}{2\pi }  \int_{|t|>T}  \frac{\Gamma(r+a_{\varepsilon}-2\varepsilon-it)(C_{\varphi}nx)^{-a_{\varepsilon}+2\varepsilon+it} }{ \Gamma(1-a_{\varepsilon}+2\varepsilon+it)}dt.
 \end{align}
 Again, using Lemma \ref{Gamma} in the second integral above, we can rewrite it as
\begin{align}
  \frac{e^{\pm\frac{i\pi r}{2}}}{2\pi}(C_{\varphi}nx)^{-a_{\varepsilon}+2\varepsilon}\int_{|t|>T} e^{\pm 2it (-\log t +1+\frac12 \log (C_{\varphi}nx))}   \left\{ t^{\delta_0-\frac 12-2\varepsilon}+O(t^{\delta_0-\frac 32-2\varepsilon})\right\}dt\nonumber.
\end{align}
Now, we apply Lemma \ref{expintegrallemma} with
\begin{align}
 F(t)=\pm 2t  \left( -\log t +1+\frac12 \log (C_{\varphi}nx) \right) \ \, \text{and}\,\,  G(t)=t^{\delta_0-\frac12-2\varepsilon},   
\end{align}
as $G(t)/F^{\prime}(t)$ is monotonic for $|t|>T$. Since $\delta_0\leq \frac 12$ and $n\leq N$, we have
\begin{align}
    \Bigg|\frac{F'(t)}{G(t)}\Bigg|=\frac{\log\left(\frac{t^2}{C_{\varphi}nx}\right)}{\left|t^{\delta_0-\frac 12-2\varepsilon}\right|}\geq \frac{\log\left(\frac{N+\frac 12}{n}\right)}{T^{\delta_0-\frac 12-2\varepsilon}}.
\end{align}
Therefore,
\begin{align}
\int_{|t|>T}  \frac{\Gamma(r+a_{\varepsilon}-2\varepsilon-it)(C_{\varphi}nx)^{-a_{\varepsilon}+2\varepsilon+it} }{ \Gamma(1-a_{\varepsilon}+2\varepsilon+it)}dt&\ll T^{\delta_0-\frac12-2\varepsilon} (C_{\varphi}nx)^{-a_{\varepsilon}+2\varepsilon}  \left( \frac{1}{\left(\log \frac{N+\frac12}{n}\right)}+1 \right)\\
&\ll T^{\delta_0-\frac12-2\varepsilon} (C_{\varphi}nx)^{-a_{\varepsilon}+2\varepsilon}.
\end{align} 
Substituting the above bound in \eqref{J_a}, we obtain
 \begin{align}\label{J_a1} 
 L_{a_{\varepsilon}}(n)&= \frac{1}{2\pi i}  \int_{(-a_{\varepsilon}+2\varepsilon)} \frac{\Gamma(r-w)(C_{\varphi}nx)^{w}}{ \Gamma(w)~w} dw+O\left(T^{\delta_0-\frac12-2\varepsilon} (C_{\varphi}nx)^{-a_{\varepsilon}+2\varepsilon} \right).
 \end{align}
 Subsequently, since $0\leq \delta_0\leq \frac 12$, we make use of the identity \cite[Eq. (2.12)]{Berndt3} involving the Bessel function $J_{\nu}(x)$ of nonnegative integral order $\nu$ to get
  \begin{align}
  L_{a_{\varepsilon}}(n)&=   (C_{\varphi}nx)^{\frac{r}{2}}J_r(2\sqrt{C_{\varphi}nx})+O\left(T^{\delta_0-\frac12-2\varepsilon} (C_{\varphi}nx)^{-a_{\varepsilon}+2\varepsilon} \right).
\end{align}
Upon substituting $L_{a_\varepsilon}(n)$ in \eqref{MEN}, we deduce that
\begin{align}
   D_{0}(x)=C_{\varphi}^{-\frac{r}{2}} x^{\frac{r}{2}}\sum_{n=1}^{N}\frac{b(n)}{ n^{\frac{r}{2}}}  &J_r(2\sqrt{C_{\varphi}nx})+\begin{cases}
        O\left(x^{\frac{\delta_0}{2}+\frac{r}{2}-\frac14+\varepsilon} N^{-\frac12}  +  x^{-\frac{\delta_0}{2}+\frac r2-\frac 14-\varepsilon}\right)&\text{if }\ \delta_0<\frac 12\\
        O\left(x^{\frac{\delta_0}{2}+\frac{r}{2}-\frac14+\varepsilon} N^{-\frac12}  +  x^{-\frac{\delta_0}{2}+\frac r2-\frac 14+\varepsilon}N^{\varepsilon}\right) &\text{if }\ \delta_0=\frac 12.
    \end{cases} \label{truncattedwithJ}
\end{align}
The classical asymptotic formula for the Bessel function \(J_\nu(x)\), as \(x\to \infty\) \cite[Eq. (1.3.15)]{Jutila}, is given by
\begin{align}
J_\nu(x)
=
\sqrt{\frac{2}{\pi x}}
\cos\left(
x-\frac{\nu\pi}{2}-\frac{\pi}{4}
\right)
+
O_\nu\!\left(x^{-\frac32}\right).
\label{eq:bessel-asymp}
\end{align}
Using this in \eqref{truncattedwithJ}, we deduce the formula \eqref{truncated1}. This completes our proof.

\begin{rem}
    The proofs of Theorems \ref{truncatedformula2} and \ref{truncatedformula3} follow the same argument as that of Theorem \ref{truncatedformula} and are therefore omitted. In the proof of Theorem \ref{truncatedformula2}, identity \cite[Eq. (2.12)]{Berndt3} is used when $p$ is a positive odd integer, whereas identities \eqref{I1int} and \eqref{I2int} are employed in all remaining cases.
\end{rem}

\section{Proof of Theorem \ref{th1}}\label{mainproof}
 In this section, we follow a construction method analogous to that in \cite{KT}. Here, we take $D_{0}$ and $r$ to be as in Theorem \ref{truncatedformula}. The proofs of the remaining two choices of $\Delta(s)$ follow the same line of reasoning.
 
Using the truncated Vorono\"{i} formula \eqref{truncated1}, we write $P(t)$ as
\begin{align}
P(t)=\frac{1}{\sqrt{\pi}}C_{\varphi}^{-\frac{r}{2}-\frac14} \sum_{n=1}^{N}&\frac{b(n)}{ n^{\frac{r}{2}+\frac14}}  \cos\left(
2\sqrt{C_{\varphi}n}~t-\frac{r \pi}{2}-\frac{\pi}{4}
\right)\\
&\hspace{0.7cm}+\begin{cases}
    O\left(t^{\delta_0+\varepsilon} N^{-\frac12}  +t^{-\delta_0-\varepsilon}  \right)&\text{if }\ \delta_0<\frac 12\\
    O\left(t^{\delta_0+\varepsilon} N^{-\frac12}  +t^{-\delta_0+\varepsilon}N^{\varepsilon}  \right)&\text{if }\ \delta_0=\frac 12 .\label{defi12}
\end{cases}    \end{align}
Let $J>1$. In view of the above formula, we define
\begin{align}
p_J(t):=\frac{1}{\sqrt{\pi}}C_{\varphi}^{-\frac{r}{2}-\frac14}\sum_{n=1}^{J}\frac{b(n)}{ n^{\frac{r}{2}+\frac14}}  \cos\left(
2\sqrt{C_{\varphi}n}~t-\frac{r \pi}{2}-\frac{\pi}{4}
\right) .
\end{align}
Let \(M\) be a sufficiently large positive real number, and assume that
\[
M\le t\le 2M.
\]
Further, let \(1<\theta<2\) and suppose that \(J<M^{\theta}\). Setting
\[
N=M^{\theta}
\]
in \eqref{defi12}, we define
\begin{align}
\mathcal{P}_1(t)&:=P(t)-p_J(t)
= S(t)+\begin{cases}
    O\!\left(M^{\delta_0-\frac{\theta}{2}+\varepsilon}+M^{-\delta_0-\varepsilon}\right)&\text{if }\ \delta_0<\frac 12\\
    O\!\left(M^{\delta_0-\frac{\theta}{2}+\varepsilon}+M^{-\delta_0+\varepsilon}\right)&\text{if }\ \delta_0=\frac 12,
\end{cases}
\end{align}
where 
\begin{align}
S(t):&=\frac{1}{\sqrt\pi}C_{\varphi}^{-\frac{r}{2}-\frac14} \sum_{J<n\le M^{\theta}}\frac{b(n)}{n^{\frac r2+\frac 14}}
\cos\!\left(2\sqrt{C_{\varphi}n}~t-\frac{r\pi}{2}-\frac{\pi}{4}\right).
\end{align}
Using the fact that $\delta_0\leq \frac12$ and $\theta>1$, we have 
\begin{align}
\mathcal{P}_1(t)=S(t)+\begin{cases}
    O\!\left(M^{-\frac12(\theta-1)+\varepsilon}+M^{-\delta_0-\varepsilon}\right)&\text{if }\ \delta_0<\frac 12\\
    O\!\left(M^{-\frac12(\theta-1)+\varepsilon}+M^{-\delta_0+\varepsilon}\right)&\text{if }\ \delta_0=\frac 12.\label{opp}
\end{cases}
\end{align}
Then since $(a+b)^4=a^4+O(|b|a^3+|b|^4)$, it follows from \eqref{opp} that if $\delta_0<\frac 12$,
\begin{align}
    \int_M^{2M}|\mathcal{P}_1(t)|^4dt&\leq \int_M^{2M}|S(t)|^4dt+O\left((M^{-\frac{1}{2}(\theta-1)+\varepsilon}+M^{-\delta_0-\varepsilon})\int_M^{2M}|S(t)|^3dt\right)\\
    &\qquad + O\left(\left( M^{-\frac{1}{2}(\theta-1)+\varepsilon}+M^{-\delta_0-\varepsilon}\right)^4M\right). \label{p1tpower4_case1}  
    \end{align}
    Likewise, if $\delta_0=\frac 12$, we get
    \begin{align}
    \int_M^{2M}|\mathcal{P}_1(t)|^4dt&\leq \int_M^{2M}|S(t)|^4dt+O\left((M^{-\frac{1}{2}(\theta-1)+\varepsilon}+M^{-\delta_0+\varepsilon})\int_M^{2M}|S(t)|^3dt\right)\\
    &\qquad + O\left(\left( M^{-\frac{1}{2}(\theta-1)+\varepsilon}+M^{-\delta_0+\varepsilon}\right)^4M\right). \label{p1tpower4_case2}  
    \end{align}
Let us first estimate $\int_M^{2M}|S(t)|^4dt$. Define
\begin{align}
f=f(n,m,k,l)
:=
 \frac{1}{\pi^2}C_{\varphi}^{-2r-1}(nmkl)^{-\frac{r}{2}-\frac14}
b(n)\overline{b(m)}b(k)\overline{b(l)},
\label{3.3}
\end{align}
for \(J<n,m,k,l\le M^{\theta}\), and \(f=0\) otherwise. We have 
\begin{align}
    |S(t)|^4=\sum f&\cos\left(2\sqrt{C_{\varphi}n}~t-\frac{r\pi}{2}-\frac{\pi}{4}\right)\cos\left(2\sqrt{C_{\varphi}m}~t-\frac{r\pi}{2}-\frac{\pi}{4}\right)\\
    & \hspace{2cm}\times\cos\left(2\sqrt{C_{\varphi}k}~t-\frac{r\pi}{2}-\frac{\pi}{4}\right)\cos\left(2\sqrt{C_{\varphi}l}~t-\frac{r\pi}{2}-\frac{\pi}{4}\right).
\end{align}
    An easy computation shows that
    \begin{align}
|S(t)|^4
=
S_1(t)+S_2(t)+S_3(t)+S_4(t)+S_5(t)+S_6(t)+S_7(t),
\label{3.4}
\end{align}
where
\begin{align}
S_1(t)&:=\frac38 \sum_{\sqrt n+\sqrt m= \sqrt k+\sqrt l}
f,\\
S_2(t)
&:=
\frac38
\sum_{\sqrt n+\sqrt m\ne \sqrt k+\sqrt l}
f
\cos\!\left(
2\sqrt{C_{\varphi}}(\sqrt n+\sqrt m-\sqrt k-\sqrt l)t
\right),
\nonumber\\
S_3(t)
&:=
\frac12
\cos (r\pi)\sum
f
\sin\!\left(
2\sqrt{C_{\varphi}}(\sqrt n+\sqrt m+\sqrt k-\sqrt l)t
\right),
\nonumber\\
S_4(t)
&:=
-\frac12
\sin (r\pi)\sum_{\sqrt{n}+\sqrt{m}+\sqrt{k}=\sqrt{l}}
f,
\nonumber\\
S_5(t)
&:=
-\frac12
\sin (r\pi)\sum_{\sqrt{n}+\sqrt{m}+\sqrt{k}\neq \sqrt{l}}
f
\cos\!\left(
2\sqrt{C_{\varphi}}(\sqrt n+\sqrt m+\sqrt k-\sqrt l)t
\right),
\nonumber\\
S_6(t)
&:=
 -\frac18 \cos (2r\pi)
\sum
f
\cos \!\left(
2\sqrt{C_{\varphi}}(\sqrt n+\sqrt m+\sqrt k+\sqrt l)t
\right),\\
S_7(t)
&:=
 -\frac18 \sin (2r\pi)
\sum
f
\sin\!\left(
2\sqrt{C_{\varphi}}(\sqrt n+\sqrt m+\sqrt k+\sqrt l)t
\right).
\end{align}

We will see that the estimation of $\int_M^{2M} S_2(t)dt$, $\int_M^{2M} S_3(t)dt$, and $\int_M^{2M} S_5(t)dt$ follows essentially the same argument. Likewise, the analysis for $\int_M^{2M} S_6(t)dt$ and $\int_M^{2M} S_7(t)dt$ are obtained by an analogous method. All distinct cases are treated separately in the subsections below.

\subsection{Estimation of $\int_M^{2M} S_1(t) dt$.}\label{ss1}
Note that for $n,m,k,l\in\mathbb{N}$, the relation
$$\sqrt{n}+\sqrt{m}=\sqrt{k}+\sqrt{l}$$
holds if and only if either $(n,m)=(k,l)$ or $(n,m)=(l,k)$, or else $n,m,k,l$ have the same square-free part. In the latter case, denoting this common square-free part by \(h\), we may write
\[
n=\alpha^2 h,\qquad
m=\beta^2 h,\qquad
k=\gamma^2 h,\qquad
l=\phi^2 h,
\]
with $\alpha+\beta=\gamma+\phi$.
Consequently,
\begin{align}
 S_1(t)&
 \ll 
\sum_{\substack{J<n,m\le M^{\theta}}} n^{-r-\frac{1}{2}} m^{-r-\frac{1}{2}} |b(n)|^2 |b(m)|^2 \\
&\hspace{3cm}+\sum_{\substack{h\ge1\\ \alpha,\beta,\gamma,\phi\ge1\\ \alpha^2 h>J}} (\alpha\beta\gamma\phi)^{-r-\frac{1}{2}} h^{-2r-1} |b(\alpha^2 h)|\,|b(\beta^2 h)|\,|b(\gamma^2 h)|\,|b(\phi^2 h)|.
\end{align}
Using the standard bound $b(n)\ll n^{\varepsilon}$, together with the assumption $r>\frac 12,$ we obtain
\begin{align}\label{S1final}
\int_M^{2M}S_1(t) dt
&\ll M
J^{1-2r+\varepsilon}
+
M\sum_{h\geq 1}h^{-2r-1+\varepsilon}\sum_{\alpha>\sqrt{\frac{J}{h}}}
\alpha^{-r-\frac{1}{2}+\varepsilon}\nonumber
\\
&\ll MJ^{1-2r+\varepsilon}+MJ^{\frac14-\frac{r}{2}+\varepsilon}
\\
&\ll MJ^{\frac14-\frac{r}{2}+\varepsilon}.
\end{align}

\subsection{Estimation of $\int_M^{2M} S_4(t) dt$.}
For $n,m,k,l\in\mathbb{N}$, the relation 
$$\sqrt{n}+\sqrt{m}+\sqrt{k}=\sqrt{l}$$
holds if and only if $n,m,k$ and $l$ all have the same square-free part, say $h$ such that $n=\alpha^2 h$, $m=\beta^2h$, $k=\gamma^2h$, $l=\phi^2 h$ and $\alpha+\beta+\gamma=\phi$. Hence,
\begin{align}
 S_4(t)&\ll \sum_{h\geq 1\atop{\alpha,\beta,\gamma,\phi\geq 1\atop{\alpha>\sqrt{\frac{J}{h}}}}}(\alpha\beta\gamma\phi)^{-r-\frac 12}h^{-2r-1}|b(\alpha^2h)||b(\beta^2h)||b(\gamma^2h)||b(\phi^2h)|\\
    &\ll \sum_{h\geq 1}h^{-2r-1+\varepsilon}\sum_{\alpha>\sqrt{\frac{J}{h}}}\alpha^{-r-\frac 12+\varepsilon}\\
    &\ll \sum_{h\geq 1}h^{-2r-1+\varepsilon}\left(\frac{J}{h}\right)^{\frac{1}{4}-\frac r2+\varepsilon}\ll J^{\frac{1}{4}-\frac r2+\varepsilon},
\end{align}
where the second step follows from $r > \frac{1}{2}$. Thus, 
\begin{align}
    \int_M^{2M}S_4(t)dt\ll MJ^{\frac{1}{4}-\frac r2+\varepsilon}.
\end{align}

\subsection{Estimation of $\int_{M}^{2M}S_2(t) dt$, $\int_{M}^{2M}S_3(t) dt$ and $\int_{M}^{2M}S_5(t) dt$}           
We estimate the integral
\[
\int_M^{2M}S_3(t)\,dt,
\]
since the corresponding estimates for the integrals involving \(S_2(t)\) and \(S_5(t)\) are obtained in an analogous manner.
Define
\begin{align}\label{POO1}
\Delta_1:=\sqrt{n}+\sqrt{m}+\sqrt{k}-\sqrt{l}.
\end{align}
Choose a real number \(c\) satisfying
\[
\frac{\theta}{2}<c<1.
\]

We distinguish two cases according to the size of \(\Delta_1\).  We use the trivial estimate
\[
\int_M^{2M}
\sin\!\left(2\sqrt{C_{\varphi}}\Delta_1 t\right)\,dt
\ll M.
\]
if $|\Delta_1|\le M^{-c}$. On the other hand, if $|\Delta_1|>M^{-c}$,
then, by the first derivative test,
\[
\int_M^{2M}
\sin\!\left(2\sqrt{C_{\varphi}}\Delta_1 t\right)\,dt
\ll |\Delta_1|^{-1}.
\]
Accordingly, we partition the integral $\int_M^{2M} S_3(t)\, dt$ into two sums determined by the magnitude of $\Delta_1$, so that
\begin{align}\label{S3division}
\int_M^{2M} S_3(t)\, dt
\ll 
\underbrace{M \sum_{0 < |\Delta_1| \le M^{-c}} |f|}_{W_1} 
+ 
\underbrace{\sum_{|\Delta_1| > M^{-c}} |f|\, |\Delta_1|^{-1}}_{W_2}.
\end{align}

First, consider $W_1$. Here, we may assume that $J<n \leq m \leq k \leq M^{\theta}$, since the sum is symmetric in $n, m,$ and $k$. In this case, the condition $|\Delta_1| \leq M^{-c}$ implies that 
\begin{align}
    \sqrt{l}& \leq  \sqrt{n} + \sqrt{m} + \sqrt{k} + |\Delta_1| \hspace{0.5cm} \text{(by using $-|\Delta_1| \leq \Delta_1$)}\\
    &\leq 4 \sqrt{k} \hspace{0.5cm} \text{(by using $n \leq m \leq k \leq M^{\theta}$ and $|\Delta_1| \leq M^{-c} \leq 1\leq \sqrt{k}$).}
\end{align}
This gives
\begin{align}
    |\Delta_1||\sqrt{n}+ \sqrt{m} + \sqrt{l} + \sqrt{k}| &= |\sqrt{n}+ \sqrt{m} + \sqrt{k} - \sqrt{l}||\sqrt{n}+ \sqrt{m} + \sqrt{k} + \sqrt{l}| \\
    &\leq 7 \sqrt{k} M^{-c}.
\end{align}
Hence,
\begin{equation}\label{eq1}
    \left|l - \left(n + m + k + 2(\sqrt{nm} + \sqrt{nk} + \sqrt{mk})\right)\right| \leq 7 \sqrt{k} M^{-c}. 
\end{equation}
Note that $7\sqrt{k} M^{-c}\leq 7M^{-(c-\frac{\theta}{2})}=o(1)$, since $\frac{\theta}{2}<c$. Also, there will be no natural number $l$ satisfying $0<|\Delta_1|\leq M^{-c}$ unless
\begin{align}\label{eq2}
    0<||2(\sqrt{nm} + \sqrt{nk} + \sqrt{mk})|| \leq 7 \sqrt{k} M^{-c}
\end{align}
holds true.
In view of \eqref{eq1} and \eqref{eq2}, there could be at most one eligible $l$ only. Moreover, such an $l$, if exists, must satisfy
$$l=m+n+k+\lfloor2\sqrt{nm}+2\sqrt{nk}+2\sqrt{mk}\rfloor.$$
From this, it is clear that $k\leq l\leq 9k$. Hence,
\begin{align}
    W_1&\ll M\sum_{J<n\leq m\leq k\leq M^{\theta} \atop{0<|\Delta_1|\leq M^{-c}}}|b(n)||\overline{b(m)}||b(k)||\overline{b(l)}|(nmkl)^{-(\frac{r}{2}+\frac14)}\\
    &\ll M^{1+\varepsilon}\sum_{J<n\leq m\leq k\leq M^{\theta} \atop{\eqref{eq2}}}|b(n)||b(m)| (nm)^{-(\frac{r}{2}+\frac14)}k^{-r-\frac12}.\label{eq2.5}
    \end{align}
Applying Lemma \ref{lem2} with $\alpha=2(\sqrt{n}+\sqrt{m})$, $\beta=2\sqrt{nm}$ and $\delta=7\sqrt{2K}M^{-c}$, we obtain
\begin{align}
    \sum_{K<k\leq 2K \atop{\eqref{eq2}}}k^{-r-\frac12} &\ll K^{-r-\frac12}\left(K^{\frac32}M^{-c}+K^{\frac12}m^{\frac16}+m^{-\frac14}K^{\frac34}\right)\\
    &=\left(K^{-r+1}M^{-c}+K^{-r}m^{\frac16}+K^{-r+\frac14}m^{-\frac14}\right).
\end{align}
Therefore, the sum in \eqref{eq2.5} is
\begin{align}
 &   \ll \sum_{J<n\leq m\leq 2K } |b(n)b(m)|(nm)^{-\frac{r}{2}-\frac14} \sum_{K<k\leq 2K \atop{\eqref{eq2}}} k^{-r-\frac12}\\
    &\ll \sum_{J<n\leq m\leq 2K }\left(K^{-r+1}M^{-c}+K^{-r}m^{\frac16}+K^{-r+\frac14}m^{-\frac14}\right)|b(n)b(m)|(nm)^{-\frac{r}{2}-\frac14} \\
   &= K^{-r+1}M^{-c}\sum_{J<n\leq m\leq 2K }|b(n)b(m)|(nm)^{-\frac{r}{2}-\frac14} +K^{-r}\sum_{J<n\leq m\leq 2K } |b(n)|n^{-\frac{r}{2}-\frac14} |b(m)|m^{-\frac{r}{2}-\frac14+\frac16}\\
   & \ \ \ \ \ +K^{-r+\frac14}\sum_{J<n\leq m\leq 2K } |b(n)b(m)|n^{-\frac{r}{2}-\frac14} m^{-\frac{r}{2}-\frac12}\\
    &\ll K^{-r+\delta_0+1+\varepsilon}M^{-c}+K^{-r+\delta_0+\frac16+\varepsilon}+K^{-r+\frac14} \underbrace{ \sum_{J<n\leq m \leq 2K} |b(n)| n^{-(\frac{r}{2}+\frac14)} |b(m)| m^{-(\frac{r}{2}+\frac12)}}_{:= \sum\limits_1}. \label{eq3}
\end{align}
To estimate the last sum above, we have
\begin{align}
\sum_1&=\sum_{J<n\leq m \leq 2K} |b(n)| n^{-(\frac{r}{2}+\frac14)} |b(m)| m^{-(\frac{r}{2}+\frac12)}\\
&\ll \sum_{J<n \leq 2K} |b(n)| n^{-(\frac{r}{2}+\frac14)}  \sum_{n\leq m \leq 2K} \frac{|b(m)|}{m^{\sigma_a+\varepsilon}}   m^{(\frac{\delta_0}{2}-\frac14+\varepsilon)}\\
&\ll K^{\varepsilon}\sum_{J<n \leq 2K} |b(n)| n^{-(\frac{r}{2}+\frac14)} n^{(\frac{\delta_0}{2}-\frac14)}\\
&\ll \begin{cases}
     K^{\delta_0-\frac 14+\varepsilon}&\text{if} \ \ \delta_0 \geq \frac 14 \\
    K^{\varepsilon}J^{\delta_0-\frac 14} &\text{if} \ \ \delta_0<\frac 14,
\end{cases}
\end{align}
where the second last step above follows by the fact that  $\delta_0\leq \frac12$. Thus,
\begin{align}
    K^{-r+\frac 14}\ \sum_1\ll  \begin{cases}
     K^{-r+\delta_0+\varepsilon}&\text{if}  \ \ \delta_0\geq \frac{1}{4} \\
    K^{-r+ \frac 14 +\varepsilon}J^{\delta_0-\frac 14}&\text{if} \ \ \delta_0<\frac 14.
\end{cases}
\end{align}
Hence, \eqref{eq3} gives
\begin{align}
    &\sum_{J<n\leq m\leq 2K \atop{K<k\leq 2K \atop{\eqref{eq2}}}} |b(n)b(m)|(nm)^{-\frac{r}{2}-\frac14}k^{-r-\frac12}\\
   &\ll \begin{cases}
     K^{-r+\delta_0+1+\varepsilon}M^{-c}+K^{-r+\delta_0+\frac16+\varepsilon}&\text{if} \ \ \ \delta_0\geq \frac 14 \\
     K^{-r+\delta_0+1+\varepsilon}M^{-c}+K^{-r+\delta_0+\frac16+\varepsilon}+K^{-r+\frac 14+\varepsilon} J^{\delta_0-\frac 14} &\text{if} \ \ \ \delta_0<\frac 14.
\end{cases}
    \end{align}
Combining the above two bounds, we get 
\begin{align}
    \sum_{J<n\leq m\leq 2K \atop{K<k\leq 2K \atop{\eqref{eq2}}}} |b(n)b(m)|(nm)^{-\frac{r}{2}-\frac14}k^{-r-\frac12} &\ll K^{-r+\delta_0+1+\varepsilon}M^{-c}+K^{-r+\delta_0+\frac16+\varepsilon}+K^{-r+\frac 14+\varepsilon} J^{\delta_0-\frac 14}\\
    &\ll K^{-r+\delta_0+\frac{1}{2}}M^{\frac{\theta}{2}-c+\varepsilon}+K^{-r+\delta_0+\frac16+\varepsilon}+K^{-r+\frac 14+\varepsilon} J^{\delta_0-\frac 14},
\label{W1finalexp}
\end{align}
where $\frac{\theta}{2} < c < 1$ and $r - \delta_0 \geq \frac{1}{2}$. Since $k\leq l\leq 9k$, using Lemma \ref{lem1}, $M^{-c}\geq |\Delta_1|>0$ implies $k^{-\frac72}\ll M^{-c}$. Substituting $K=\frac{M^{\theta}}{2^i}$ in \eqref{W1finalexp}, and summing over the all dyadic intervals, \eqref{eq2.5} gives
\begin{align}
    W_1&\ll  M^{1 + \varepsilon}\left(M^{\frac{2c}{7}\left(-r+\delta_0+\frac 12\right)+\left(\frac{\theta}{2}-c+\varepsilon\right)} + M^{\frac{2c}{7}\left(-r+\delta_0+\frac 16+\varepsilon\right)} + M^{\frac{2c}{7}\left(-r+\frac 14+\varepsilon\right)}J^{\delta_0 - \frac{1}{4}} \right),\label{ww1}
   \end{align}
  using the fact that \begin{align}\label{mk1est}
  M^{\theta}\gg K\gg k\gg M^{\frac{2c}{7}}.
\end{align}  

Our next goal is to estimate $W_2$. We decompose the sum into the following two parts:
\begin{align}
    W_2=\bigg(\underbrace{\sum_{M^{-c}<|\Delta_1|\leq k^{\alpha}}}_{W_{21}}+\underbrace{\sum_{|\Delta_1|>k^{\alpha}}}_{W_{22}}\bigg)|f||\Delta_1|^{-1},
\end{align}
where $\alpha$ is a real number satisfying
\begin{align}
    \alpha<\frac12 \ \ \  \ \mbox{and}\ \ \  \ \theta(2\delta_0-\alpha)\leq \theta(1-\alpha)<c.
\end{align}
Without loss of generality, we may assume that $J<n\leq m\leq k\leq M^{\theta}$. Let us first estimate $W_{21}$. From the fact
\begin{align}
    l=(\sqrt{n}+\sqrt{m}+\sqrt{k}-\Delta_1)^2 = (\sqrt{n}+\sqrt{m}+\sqrt{k})^2+O(|\Delta_1|\sqrt{k}),
\end{align}
we can infer that $l \gg k$.
Moreover, since 
$$l-(\sqrt{n}+\sqrt{m}+\sqrt{k})^2=O(|\Delta_1|\sqrt{k})\ll k,$$ 
there are $\ll |\Delta_1|\sqrt{k}+1$ such $l$. Thus,
\begin{align}
    W_{21}&\ll M^{\varepsilon}\sum_{J<n\leq m\leq k\leq M^{\theta}\atop{M^{-c}<|\Delta_1|\leq k^{\alpha}}}|b(m)||b(n)|| b(k)|(nmk)^{-\frac{r}{2}-\frac14}k^{-\frac{r}{2}-\frac14}|\Delta_1|^{-1}(|\Delta_1|\sqrt{k}+1)\\
    &\ll M^{\varepsilon}\sum_{J<n\leq m\leq k\leq M^{\theta}\atop{M^{-c}<|\Delta_1|}}|b(m)||b(n)|| b(k)|(nm)^{-\frac{r}{2}-\frac14}k^{-r-\frac12}(\sqrt{k}+M^{c})\\
    &\ll M^{\varepsilon}\underbrace{\sum_{J<n\leq m\leq k\leq M^{\theta}}|b(m)||b(n)|| b(k)|(nm)^{-\frac{r}{2}-\frac14}k^{-r}}_{:=\sum\limits_{211}}\\
    &\hspace{3cm}+M^{c+ \varepsilon}\underbrace{\sum_{J<n\leq m\leq k\leq M^{\theta}} |b(m)||b(n)|| b(k)|(nm)^{-\frac{r}{2}-\frac14}k^{-r-\frac12}}_{:=\sum\limits_{212}}.\label{W21bound}
\end{align}

Now, we estimates the sums $\sum\limits_{211}$ and $\sum\limits_{212}$, given in \eqref{W21bound}.  Consider first
\begin{align}
  \sum_{211}  &\ll \sum_{J<n\leq m\leq M^{\theta}}|b(m)||b(n)| (nm)^{-\frac{r}{2}-\frac14} \sum_{m\leq k\leq M^{\theta}} | b(k)|k^{-r} \nonumber\\
  &\ll \sum_{J<n\leq m\leq M^{\theta}}|b(m)||b(n)| (nm)^{-\frac{r}{2}-\frac14} \sum_{m\leq k\leq M^{\theta}} \frac{| b(k)|}{k^{\sigma_a+\varepsilon}}k^{\frac{\delta_0}{2} -\frac12(r-\frac12)+\varepsilon} \nonumber\\
  &\ll M^{\frac{\theta }{4}+\varepsilon} \sum_{J<n\leq m\leq M^{\theta}}|b(m)||b(n)| n^{-\frac{r}{2}-\frac14} m^{\frac{\delta_0}{2}-r-\frac 14}\nonumber\\
  & \ll M^{\frac{\theta }{4}+\varepsilon} \sum_{J<n\leq M^{\theta}}|b(n)| n^{-\frac{r}{2}-\frac14} \sum_{n\leq m\leq M^{\theta}} \frac{|b(m)|}{m^{\sigma_a+\varepsilon}}m^{\delta_0-\frac{r}{2}+\varepsilon},
\end{align}
where the third step follows by the fact $r>\delta_0$. Hence,
\begin{align}\label{eq211}
 \sum_{211}  \ll \begin{cases}
     M^{\theta\left(\frac{1 }{4}+\frac{3\delta_0}{2}-\frac{r}{2}\right)+\varepsilon}&\mbox{if} \ \ \delta_0\geq \frac{r}{2}\\
      M^{\theta\left(\frac{1 }{4}+\frac{\delta_0}{2}\right)+\varepsilon}J^{\delta_0-\frac{r}{2}} &\mbox{if} \ \ \delta_0< \frac{r}{2}.\\
 \end{cases}    
\end{align}
Next, we estimate
\begin{align}
\sum_{212} &\ll \sum_{J<n\leq m\leq M^{\theta}}|b(m)||b(n)| (nm)^{-\frac{r}{2}-\frac14} \sum_{m\leq k\leq M^{\theta}} | b(k)|k^{-r-\frac12} \nonumber\\   
&\ll \sum_{J<n\leq m\leq M^{\theta}}|b(m)||b(n)| (nm)^{-\frac{r}{2}-\frac14} \sum_{m\leq k\leq M^{\theta}} \frac{| b(k)|}{k^{\sigma_a+\varepsilon}} k^{\frac{\delta_0}{2}-\frac{r}{2}-\frac14+\varepsilon} \nonumber\\ 
&\ll \sum_{J<n\leq M^{\theta}}|b(n)| n^{-\frac{r}{2}-\frac14} \sum_{n\leq m\leq M^{\theta}} \frac{| b(m)|}{m^{\sigma_a+\varepsilon}} m^{\delta_0-\frac{r}{2}-\frac14+\varepsilon},
\end{align}
since $r> \frac12\geq \delta_0$. This gives that 
\begin{align}
 \sum_{212} &\ll   M^{\theta(\delta_0-\frac{r}{2}+\varepsilon)} \sum_{J<n\leq M^{\theta}}|b(n)| n^{-\frac{r}{2}-\frac12} \\ &\ll M^{\theta(\delta_0-\frac{r}{2}+\varepsilon)} \sum_{J<n\leq M^{\theta}}\frac{|b(n)|}{n^{\sigma_a+\varepsilon}} n^{\frac{\delta_0}{2}-\frac14} \\ &\ll M^{\theta(\delta_0-\frac{r}{2}+\varepsilon)} J^{\frac{\delta_0}{2}-\frac14},\label{sum2121}
\end{align}
whenever $\delta_0\geq \frac{r}{2}$, and the last step follows because $\delta_0\leq \frac12$. Similarly, if $\delta_0<\frac{r}{2}$, then we obtain  
\begin{align}
 \sum_{212} \ll   \sum_{J<n\leq M^{\theta}}\frac{|b(n)|}{n^{\sigma_a+\varepsilon}} n^{\frac{3\delta_0}{2}-\frac{r}{2}-\frac14+\varepsilon}  \ll J^{\frac{3\delta_0}{2}-\frac{r}{2}-\frac14+\varepsilon}.\label{sum2122}
\end{align}
Hence, by combining \eqref{sum2121} and \eqref{sum2122}, we get
\begin{align}\label{eq212}
 \sum_{212} \ll
 \begin{cases}
  M^{\theta(\delta_0-\frac{r}{2}+\varepsilon)} J^{\frac{\delta_0}{2}-\frac14}&\mbox{if} \ \ \delta_0\geq \frac{r}{2}\\
 J^{\frac{3\delta_0}{2}-\frac{r}{2}-\frac14+\varepsilon}&\mbox{if}  \ \ \delta_0< \frac{r}{2}.
 \end{cases}  
 \end{align}
Therefore, \eqref{eq211} and \eqref{eq212} together with \eqref{W21bound} give
\begin{align}
W_{21}\ll    \begin{cases}
    M^{\theta\left(\frac{1 }{4}+\frac{3\delta_0}{2}-\frac{r}{2}\right)+\varepsilon} +M^{c+\theta(\delta_0-\frac{r}{2})+\varepsilon} J^{\frac{\delta_0}{2}-\frac14}&\mbox{if} \  \ \delta_0\geq \frac{r}{2}\\
   M^{\theta\left(\frac{1 }{4}+\frac{\delta_0}{2}\right) +\varepsilon}J^{\delta_0-\frac{r}{2}}+M^{c+\varepsilon}J^{\frac{3\delta_0}{2}-\frac{r}{2}-\frac14+\varepsilon}&\mbox{if}  \ \ \delta_0< \frac{r}{2}.\\
\end{cases}\label{W21finalbound}
\end{align}

Now, we evaluate $W_{22}$, which is given by 
\begin{align}
    W_{22}&\ll \sum_{k^{\alpha}<|\Delta_1|\atop{J<n\leq m\leq k\leq M^{\theta}\atop{J<l\leq M^{\theta}}}}|b(m)||b(n)|| b(k)||b(l)|(nmkl)^{-\frac{r}{2}-\frac14} k^{-\alpha}\\
    &\ll \sum_{k^{\alpha}<|\Delta_1|\atop{J<n\leq m\leq k\leq M^{\theta}\atop{J<l\leq M^{\theta}}}} |b(m)||b(n)|| b(k)||b(l)|(nml)^{-\frac{r}{2}-\frac14}k^{-\frac{r}{2}-\frac14-\alpha}\\
    &\ll \sum_{J<n\leq m\leq k\leq M^{\theta}} |b(m)||b(n)|| b(k)|(nm)^{-\frac{r}{2}-\frac14}k^{-\frac{r}{2}-\frac14-\alpha} \sum_{J<l\leq M^{\theta}}|b(l)| l^{-\frac{r}{2}-\frac14}\\
    &\ll M^{\frac{\theta~\delta_0}{2}+\varepsilon} \sum_{J<n\leq m\leq k\leq M^{\theta}} |b(m)||b(n)|| b(k)|(nm)^{-\frac{r}{2}-\frac14}k^{-\frac{r}{2}-\frac14-\alpha}.
    \end{align}
Using the fact that $\alpha < \frac{1}{2}$,  the above inequality yields 
    \begin{align}
    W_{22}&\ll M^{\frac{\theta~\delta_0}{2}+\varepsilon} \sum_{J<n\leq m\leq M^{\theta}} |b(m)||b(n)|(nm)^{-\frac{r}{2}-\frac14}\left(m^{\frac{\delta_0}{2}-\alpha+\varepsilon}+M^{\theta(\frac{\delta_0}{2}-\alpha)+\varepsilon}\right) \\
    &\ll M^{\frac{\theta~\delta_0}{2}+\varepsilon} \sum_{J<n\leq M^{\theta}} |b(n)|n^{-\frac{r}{2}-\frac14}\left(n^{(\delta_0-\alpha+\varepsilon)}+M^{\theta(\delta_0-\alpha)+\varepsilon}\right)+M^{\theta(2\delta_0-\alpha)+\varepsilon} \\
    &\ll M^{\frac{\theta~\delta_0}{2}+\varepsilon} J^{\frac{3\delta_0}{2}-\alpha}+ M^{\theta(2\delta_0-\alpha) + \varepsilon}.\label{W22bound}
\end{align}
Therefore, combining \eqref{W21finalbound} and \eqref{W22bound}, we obtain that $W_2$ is
\begin{align}
    \ll \begin{cases}
    M^{\theta\left(\frac{1 }{4}+\frac{3\delta_0}{2}-\frac{r}{2}\right)+\varepsilon} +M^{c+\theta(\delta_0-\frac{r}{2})+\varepsilon} J^{\frac{\delta_0}{2}-\frac14}+M^{\frac{\delta_0}{2}\theta+\varepsilon}J^{\frac{3\delta_0}{2}-\alpha}+ M^{\theta(2\delta_0-\alpha) + \varepsilon}&\mbox{if} \  \ \delta_0\geq \frac{r}{2}\\
   M^{\theta\left(\frac{1 }{4}+\frac{\delta_0}{2}\right) +\varepsilon}J^{\delta_0-\frac{r}{2}}+M^{c+\varepsilon}J^{\frac{3\delta_0}{2}-\frac{r}{2}-\frac14}+M^{\frac{\delta_0}{2}\theta+\varepsilon}J^{\frac{3\delta_0}{2}-\alpha}+ M^{\theta(2\delta_0-\alpha) + \varepsilon}&\mbox{if}  \ \ \delta_0< \frac{r}{2}.\\
\end{cases} 
\end{align}
This can be rewritten as 
\begin{align}
W_2\ll \begin{cases}
    M^{\theta\left(\frac{1 }{4}+\frac{3\delta_0}{2}-\frac{r}{2}\right)+\varepsilon} +M^{c+\theta(\delta_0-\frac{r}{2})+\varepsilon} J^{\frac{\delta_0}{2}-\frac14} + M^{c+\varepsilon}J^{\frac{3\delta_0}{2} - \alpha} + M^{c+\varepsilon}&\mbox{if} \  \ \delta_0\geq \frac{r}{2}\\
M^{c+\varepsilon}+M^{c+\varepsilon}J^{\frac{3\delta_0}{2}-\alpha} \label{W2final}
&\mbox{if}  \ \ \delta_0< \frac{r}{2},
\end{cases}
\end{align}
where the inequality in second case above follows by the fact $\delta_0<\frac{r}{2}$, $\theta\left(\frac{1 }{4}+\frac{\delta_0}{2}\right)<\frac{\theta}{2}<c$, $\frac{3\delta_0}{2}-\frac{r}{2}-\frac14<0$, $\frac{\theta\delta_0}{2}<\frac{\theta}{2}<c$, and $\theta(2\delta_0-\alpha)\leq \theta(1-\alpha)<c$.
By \eqref{S3division}, \eqref{ww1} and \eqref{W2final}, we get the bound for $\int_M^{2M}S_3(t
)dt$ in the following two cases:\\

\textbf{Case (i):} Suppose that $\delta_0\geq \frac r2$. We have
\begin{align}
    \int_M^{2M}S_3(t)dt\ll& ~M^{1 + \varepsilon}\left(M^{\frac{2c}{7}\left(-r+\delta_0+\frac 12\right)+\left(\frac{\theta}{2}-c+\varepsilon\right)} + M^{\frac{2c}{7}\left(-r+\delta_0+\frac 16+\varepsilon\right)} + M^{\frac{2c}{7}\left(-r+\frac 14+\varepsilon\right)}J^{\delta_0 - \frac{1}{4}} \right)\\
    &\qquad+M^{\theta\left(\frac{1 }{4}+\frac{3\delta_0}{2}-\frac{r}{2}\right)+\varepsilon} +M^{c+\theta(\delta_0-\frac{r}{2})+\varepsilon} J^{\frac{\delta_0}{2}-\frac14} + M^{c+\varepsilon}J^{\frac{3\delta_0}{2} - \alpha} + M^{c+\varepsilon}.\label{S3case1}
\end{align}
Note that all exponents of $M$ appearing on the right-hand side above are strictly less than $1$ because
$$\theta\left(\frac 14+\frac{3\delta_0}{2}-\frac {r}{2}\right)\leq\theta\left(\frac 14+\delta_0+\left(\frac{\delta_0}{2}-\frac r2\right)\right)<\theta\delta_0\leq \frac{\theta}{2}<1,$$
and
$$c+\theta\left(\delta_0-\frac{r}{2}\right)=c+\frac{\theta\delta_0 }{2}+\frac{\theta}{2}\left(\delta_0-r\right)\leq c+\frac{\theta}{4}-\frac{\theta}{4}=c <1.$$

\textbf{Case (ii):} Let $\delta_0<\frac r2$. Then 
\begin{align}
    \int_M^{2M}S_3(t)dt & \ll M^{1 + \varepsilon}\left(M^{\frac{2c}{7}\left(-r+\delta_0+\frac 12\right)+\left(\frac{\theta}{2}-c+\varepsilon\right)} + M^{\frac{2c}{7}\left(-r+\delta_0+\frac 16+\varepsilon\right)} + M^{\frac{2c}{7}\left(-r+\frac 14+\varepsilon\right)}J^{\delta_0 - \frac{1}{4}} \right)\\
    & \qquad +M^{c+\varepsilon}+M^{c+\varepsilon}J^{\frac{3\delta_0}{2}-\alpha}.\label{S3case2}
    \end{align}

Same bounds can be established for $\int_{M}^{2M}S_2(t)dt$ and $\int_M^{2M}S_5(t)dt$.

\subsection{Estimation of $\int_M^{2M} S_6(t) dt$ and $\int_M^{2M} S_7(t) dt$}\label{ss4}  We estimate the bound for $\int_M^{2M} S_6(t) dt$ here. The estimation of the corresponding integral involving $S_7(t)$ proceeds along similar lines.

Without loss of generality, we may assume that $J<n\leq m\leq k\leq l\leq M^{\theta}$ in this case.
\begin{align}
    \int_M^{2M} S_6(t) dt &\ll  \sum_{J<n\leq m\leq k\leq l\leq M^{\theta}} \dfrac{|b(m)||b(n)|| b(k)||b(l)|}{(nmkl)^{\frac{r}{2}+\frac14}}\cdot \dfrac{1}{\sqrt{n} + \sqrt{m} + \sqrt{k} + \sqrt{l}} \\ 
    &\ll  \sum_{J<n\leq m\leq k\leq M^{\theta}} \dfrac{|b(m)||b(n)|| b(k)|}{(nmk)^{\frac{r}{2}+\frac14}} \sum_{k\leq l\leq M^{\theta}} \dfrac{|b(l)|}{l^{\frac{r}{2}+\frac34}} \hspace{0.4cm}\\
    &\ll  M^{\varepsilon}\sum_{J<n\leq m\leq M^{\theta}} \dfrac{|b(m)||b(n)|}{(nm)^{\frac{r}{2}+\frac14}} \sum_{m \leq k \leq M^\theta}\frac{|b(k)|}{k^{\sigma_a + \varepsilon}} k^{\delta_0-\frac12+\varepsilon} \hspace{0.4cm}\\
    &\ll  M^{\varepsilon}\sum_{J<n\leq m\leq k\leq M^{\theta}} \dfrac{|b(n)|}{n^{\frac{r}{2}+\frac14}} \sum_{n \leq m \leq M^\theta} \frac{|b(m)|}{m^{\sigma_a + \varepsilon}} m^{\frac{3\delta_0}{2}-\frac12+\varepsilon}.
\end{align}
For $\delta_0 \leq \frac{1}{2}$, the exponent of $m$  in the above inequality, namely  $\frac{3\delta_0}{2} - \frac{1}{2}$, can be either positive or negative. So, the above inequality implies
\begin{align}
      \int_M^{2M} S_6(t) dt  &\ll  M^{\varepsilon}\sum_{J<n\leq M^{\theta}} \dfrac{|b(n)|}{n^{\frac{r}{2}+\frac14}} \left(n^{\frac{3\delta_0}{2}-\frac12+\varepsilon}+M^{\theta(\frac{3\delta_0}{2}-\frac12)+\varepsilon}\right)\\
    &\ll J^{(2\delta_0-\frac12)+\varepsilon}+ M^{\theta(2\delta_0-\frac12)+\varepsilon}. \label{S6estimation}
\end{align}
Here, 
$$\theta\left(2\delta_0-\frac 12\right)\leq \frac{\theta}{2}<1.$$
On similar lines, we can show that
\begin{align}
    \int_M^{2M} S_7(t) dt 
    &\ll J^{(2\delta_0-\frac12)+\varepsilon}+ M^{\theta(2\delta_0-\frac12)+\varepsilon}.\label{S7estimation}
\end{align}
By Equation \eqref{3.4}, we have
\begin{align}
    \int_M^{2M}|S(t)|^{4}dt=\sum_{i=1}^7 \left( \int_M^{2M}S_i(t)dt \right).
\end{align}
Clearly, from the estimations of integrals involving $S_i(t)$ given in Subsections \ref{ss1}-\ref{ss4}, we deduce the following:\\

\textbf{Case (i):} For $\delta_0\geq \frac r2$, we have
\begin{align}
    \int_M^{2M}|S(t)|^4dt&\ll ~MJ^{\frac 14-\frac r2+\varepsilon}+M^{1+\frac{2c}{7}\left(-r+\delta_0+\frac 12\right)+\left(\frac{\theta}{2}-c\right)+\varepsilon} + M^{1+\frac{2c}{7}\left(-r+\delta_0+\frac 16\right)+\varepsilon} \\ & \qquad+ M^{1+\frac{2c}{7}\left(-r+\frac 14\right)+\varepsilon}J^{\delta_0 - \frac{1}{4}}+M^{\theta\left(\frac{1 }{4}+\frac{3\delta_0}{2}-\frac{r}{2}\right)+\varepsilon} +M^{c+\theta(\delta_0-\frac{r}{2})+\varepsilon} J^{\frac{\delta_0}{2}-\frac14} \\ & \qquad+ M^{c+\varepsilon}J^{\frac{3\delta_0}{2} - \alpha}+ M^{c+\varepsilon}+J^{(2\delta_0-\frac12)+\varepsilon}+ M^{\theta(2\delta_0-\frac12)+\varepsilon}\label{S^4 bound with J1}.
\end{align}

\textbf{Case (ii):} For $\delta_0<\frac r2$, we get
\begin{align}
    \int_M^{2M}|S(t)|^4dt&\ll ~MJ^{\frac 14-\frac r2+\varepsilon}+M^{1+\frac{2c}{7}\left(-r+\delta_0+\frac 12\right)+\left(\frac{\theta}{2}-c\right)+\varepsilon} + M^{1+\frac{2c}{7}\left(-r+\delta_0+\frac 16\right) +\varepsilon}
    \\ & \qquad + M^{1+\frac{2c}{7}\left(-r+\frac 14\right)+\varepsilon}J^{\delta_0 - \frac{1}{4}}  +M^{c+\varepsilon} +M^{c+\varepsilon}J^{\frac{3\delta_0}{2}-\alpha} 
    + J^{(2\delta_0-\frac12)+\varepsilon} \\
    & \qquad+ M^{\theta(2\delta_0-\frac12)+\varepsilon}\label{S^4 bound with J2}.
\end{align}
\subsection{Estimation of $\int_M^{2M} |S(t)|^2 dt$} For real numbers  
$\alpha_i \neq \alpha_j$, and $\delta$, we have 
\begin{align}
\int_M^{2M}\cos(\alpha_i t+\delta)\cos(\alpha_j t+\delta)\,dt
\ll \min\!\left\{\bigl||\alpha_i|-|\alpha_j|\bigr|^{-1},\,M\right\}.
\end{align}
Using this with the bound $b(n)\ll n^{\varepsilon}$, we obtain
\begin{align}\label{S2estimation}
\int_M^{2M}|S(t)|^2\,dt
\ll
M\sum_{J<n\le M^{\theta}} n^{-\left(r+\frac{1}{2}\right)+\varepsilon}
+
\sum_{J<m<n\le M^{\theta}}
m^{-\left(r+\frac{1}{2}\right)+\varepsilon}n^{\varepsilon}(\sqrt n-\sqrt m)^{-1}.
\end{align}
The first term on the right-hand side of the above inequality 
is estimated as
\begin{align}\label{S2ft1}
M\sum_{J<n\le M^{\theta}} n^{-\left(r+\frac{1}{2}\right)+\varepsilon}
\ll M J^{\frac{1}{2}-r+\varepsilon}.
\end{align}
For the second term, we write
\begin{align}
&\sum_{J<m<n\le M^{\theta}}
m^{-\left(r+\frac{1}{2}\right)+\varepsilon}n^{\varepsilon}(\sqrt n-\sqrt m)^{-1}\\
&\ll
\sum_{J<m<n\le M^{\theta}}
m^{-\left(r+\frac{1}{2}\right)+\varepsilon}n^{\varepsilon}\frac{\sqrt n+\sqrt m}{n-m} \nonumber\\
&\ll
\sum_{J<m\le M^{\theta}}
m^{-\left(r+\frac{1}{2}\right)+\varepsilon}
\left(
\sum_{m<n\le 2m}\frac{M^{\frac{\theta}{2} + \varepsilon}}{n-m}
+
\sum_{2m<n\le M^{\theta}} n^{-\frac{1}{2}+ \varepsilon}
\right) \nonumber\\
&\ll
M^{\frac{\theta}{2} + \varepsilon}\sum_{J<m\le M^{\theta}} m^{-\left(r+\frac{1}{2}\right)+\varepsilon} \nonumber\\
&\ll
M^{\frac{\theta}{2} + \varepsilon} J^{\frac 12-r+\varepsilon}. \label{firstintegralestimate}
\end{align}
By \eqref{S2estimation}, \eqref{S2ft1} and \eqref{firstintegralestimate}, we conclude that
\begin{align}
\int_M^{2M}|S(t)|^2\,dt
\ll
M J^{\frac 12-r+\varepsilon}. \label{S(t)^2bound}
\end{align}

\subsection{Estimation of $\int_M^{2M}|S(t)|^3dt$}\label{ss6} By using Cauchy-Schwarz inequality, we have
\begin{align}\label{st324}
    \int_M^{2M} |S(t)|^3 dt \ll \left( \int_M^{2M} |S(t)|^2 dt \right)^{\frac 12} \left(\int_M^{2M} |S(t)|^4 dt \right)^{\frac 12}.
\end{align}
Using \eqref{S^4 bound with J1} for $\delta_0\geq \frac r2$ and \eqref{S^4 bound with J2} for $\delta_0<\frac r2$, together with \eqref{S(t)^2bound}, the above inequality implies that
\begin{align}
     \int_M^{2M} |S(t)|^3 dt 
    &\ll MJ^{\frac 14-\frac r2+\varepsilon}\left( 1+M^{\frac{c}{7}\left(-r+\frac 14\right)+\varepsilon}J^{\frac{\delta_0}{2} - \frac{1}{8}}+ M^{\frac{(c-1)}{2}+\varepsilon}J^{\frac{3\delta_0}{4}-\frac{\alpha}{2}}+M^{-\frac 12}J^{\delta_0-\frac 14+\varepsilon}\right).
    \label{st3r2less}
\end{align}
\\
Using \eqref{S^4 bound with J1} for $\delta_0\geq\frac r2$, $\delta_0\neq \frac 12$ and \eqref{S^4 bound with J2} for $\delta_0<\frac r2$, along with \eqref{st3r2less} in inequality \eqref{p1tpower4_case1}, substituting $M=\frac{X}{2^j}$ and summing over all the dyadic intervals with $j\in\mathbb{N}$, we obtain that
\begin{align}
   \limsup_{X \rightarrow \infty} \frac{1}{X}\int_1^X |\mathcal{P}_1(t)|^4 dt \ll J^{\frac{1}{4} - \frac{r}{2} + \epsilon}, 
\end{align}
which tends to zero whenever $J \rightarrow \infty$, since $r > \frac{1}{2}$. In case $\delta_0 = \frac{1}{2}$, we use \eqref{p1tpower4_case2} to get the same conclusion. This completes the proof of Theorem \ref{th1}.

\medskip

\begin{proof}[Proof of Corollary \ref{cor:fourth_moment}]
By Theorem \ref{th1}, the function $P(t)$ belongs to the Besicovitch space $B^4$, and its approximating trigonometric polynomials $p_j(t)$ satisfy 
\[
\|P(t) - p_j(t)\|_{4} \to 0 \quad \text{as } j \to \infty.
\]
Using the continuity of the norm $\|\cdot\|_{4}$, we obtain
\begin{align}
\|P\|_{4}^4 &= \lim_{j \to \infty} \|p_j(t)\|_{4}^4 \\
&= \frac{3}{8\pi^2} C_{\varphi}^{-2r-1} \sum_{\substack{n,m,k,l \ge 1 \\ \sqrt{n}+\sqrt{m} = \sqrt{k}+\sqrt{l}}} \frac{b(n)\overline{b(m)}b(k)\overline{b(l)}}{(nmkl)^{\frac{r}{2}+\frac{1}{4}}} \\
&\quad - \frac{\sin(r\pi)}{2\pi^2} C_{\varphi}^{-2r-1} \sum_{\substack{n,m,k,l \ge 1 \\ \sqrt{n}+\sqrt{m}+\sqrt{k} = \sqrt{l}}} \frac{b(n)\overline{b(m)}b(k)\overline{b(l)}}{(nmkl)^{\frac{r}{2}+\frac{1}{4}}},
\end{align}
which completes the proof.
\end{proof}

\section{Examples}\label{examples}
\subsection{The Classical Divisor Problem}\label{example1}
Let $b(n)=d(n)$, where $d(n)$ denotes the classical divisor function, and let $\mu_n=\pi n$. The associated Dirichlet series is given by
$$\varphi(s)=\pi^{-s}\zeta^{2}(s),$$
where $\zeta(s)$ is the Riemann zeta function. It is well known that $\zeta^{2}(s)$ satisfies the functional equation
$$\pi^{-s}\Gamma^2\left(\frac s2\right)\zeta^2(s)=\pi^{-(1-s)}\Gamma^2\left(\frac{1-s}{2}\right)\zeta^2(1-s),$$
which implies that $\varphi(s)$ satisfies
$$\Gamma^2\left(\frac s2\right)\varphi(s)=\Gamma^2\left(\frac{1-s}{2}\right)\varphi(1-s).$$
This is precisely the functional equation required in Theorem \ref{truncatedformula2} with the parameter $p=0$. Consequently, by Theorem \ref{th1}, it follows that the error term $\Delta(x)$ associated to the Dirichlet divisor problem, normalized by the factor $x^{\frac 14}$, belongs to the Besicovitch space $B^{q}$ for each $q=1,2,3,4$. 

\subsection{Epstein Zeta Functions and the Gauss Circle Problem}\label{example2}
For a $m\times m$ positive definite matrix $Q$ with integer entries, the quadratic form associated to $Q$ is defined by $Q[\textbf{x}]=\textbf{x}^tQ\textbf{x}$ for $\textbf{x}\in\mathbb{R}^n$. For $n\in\mathbb{Z}_{\geq 0}$, let $r_{Q}(n)$ denote the number of $x\in\mathbb{Z}^n$ satisfying $Q[\textbf{x}]=n$, i.e. the number of representations of $n$ by $Q$. Then the Epstein zeta function attached to $Q$ is given by
\begin{align}
    \zeta(s;Q)=\sum_{n=1}^{\infty}\frac{r_Q(n)}{n^s}=\sum_{0\neq\textbf{x}\in\mathbb{Z}^m}Q[\textbf{x}]^{-s},\  \ \ \text{for } \ \Re(s)>\frac m2.
\end{align}
This was introduced by Epstein \cite{Epstein}, who also proved that $\zeta(s;Q)$ satisfies the functional equation
\begin{align}
    \pi^{-s}\Gamma(s)\zeta(s;Q)=\sqrt{\det Q}~\pi^{s-\frac m2}\Gamma\left(\frac m2-s\right)\zeta\left(\frac m2-s;Q^{-1}\right).\label{epsteinfnal}
\end{align}
In case $Q\in\text{SL}_m(\mathbb{Z})$, $\zeta(s;Q)=\zeta(s;Q^{-1})$. For such $Q$, if $b(n)=r_Q(n)$ and $\mu_n=\pi n$, then the associated Dirichlet series is given by $$\varphi(s)=\pi^{-s}\zeta(s;Q).$$ 
If the condition $r_Q(n)\ll n^{\varepsilon}$ holds and $m=2$, then $\varphi(s)$ belongs to the class of Dirichlet series considered in Theorem \ref{truncatedformula} with $r=1$.
By Theorem \ref{th1}, we infer that $x^{-\frac 14}\mathcal{E}(x)$ is $\mathcal{B}^4$-almost periodic, where $\mathcal{E}(x)$ is the error term associated to the sum $\sideset{}{'}\sum_{n\leq x}r_Q(n)$.

The arithmetical function $r(n)$, which counts the number of representations of $n$ as a sum of two integer squares, is a special case of the general quadratic form obtained by taking $Q=\mathbf{I}_2$, the $2\times2$ identity matrix. By Jacobi's formula \cite[Theorem 3.2.1]{Berndtbook}, we have 
$$r(n)=4\sum_{d|n}\chi_4(d),$$
where  $\chi_4$ denotes the non-principal character modulo $4$, i.e.
$$\chi_4(n)=\begin{cases}
    (-1)^{\frac{n-1}{2}}\ \  &\text{if } \ n\equiv 1 ~(\bmod2)\\
    0\ \ &\text{if }\ n\equiv 0 ~(\bmod 2).
\end{cases}$$
Therefore, if $b(n)=r(n)\leq 4~d(n)\ll n^{\varepsilon}$ and $\mu_n=\pi n$, then $\varphi(s)$ can be written as
$$\varphi(s)=4\pi^{-s}~\zeta(s)L(s,\chi_4)=\pi^{-s}\zeta(s;\textbf{1}_2),$$
where $L(s,\chi_4)$ is the Dirichlet \(L\)-function associated with $\chi_4$. By equation \eqref{epsteinfnal}, $\varphi$ satisfies the functional equation
$$\Gamma(s)\varphi(s)=\Gamma(1-s)\varphi(1-s).$$
Hence, $x^{-\frac 14}\mathcal{P}(x)$ is $\mathcal{B}^4$-almost periodic, where $\mathcal{P}(x)$ is the error term associated to the Gauss circle problem. 

Note that $\zeta(s;\textbf{I}_2)=4~\zeta(s)L(s,\chi_4)$ is the Dedekind zeta function of the Gaussian field $\mathbb{Q}(i)$. 

\subsection{Exponential Divisor Sums}
Let $h/k$ be a fixed rational such that $(h,k)=1$, $h^2\equiv 1~(\bmod k)$ and $k\geq 1$. Incorporating the exponential factor $e\left(nh/k\right)=e^{\frac{2\pi i nh}{k}}$ to the divisor problem considered in Example \ref{example1}, we define the summatory function 
\begin{align}
    A_0(x)=\sideset{}{'}\sum_{n\leq x}d(n)e(nh/k),\label{exponentialsum}
\end{align}
and the corresponding Dirichlet series as
$$E(s,h/k)=\sum_{n=1}^{\infty}\frac{d(n)e(nh/k)}{n^s}.$$
This exponential sum was studied by Jutila \cite{Jutila}, who proved that $E(s,h/k)$ satisfies the functional equation
$$\left(\frac{\pi}{k}\right)^{-s}\Gamma^2\left(\frac s2\right)E(s,h/k)=\left(\frac{\pi}{k}\right)^{-(1-s)}\Gamma^2\left(\frac{1-s}{2}\right)E(1-s,h/k).$$
Taking $b(n)=d(n)e(nh/k)$ and $\mu_n=\pi n/k$,  $\varphi(s)$ belongs to the class of Dirichlet series considered in Theorem \ref{truncatedformula2} for $p=0$, similar to the case of function $d(n)$. Therefore, Theorem \ref{th1} implies that $x^{-\frac 14}\Delta_0(x)$ belongs to the Besicovitch space $B^q$ for $q=1,2,3,4$, where $\Delta_0(x)$ denotes the error term associated with the sum $A_0(x)$ in \eqref{exponentialsum} (for explicit formula, see \cite[Equation (1.5.3)]{Jutila}). 

Similar conclusions can be drawn if $d(n)$ in \eqref{exponentialsum} is replaced by the function $r(n)$, as defined in Section \ref{example2}.

\subsection{Dedekind Zeta Functions of Quadratic Fields}
Let $K$ be an algebraic number field with $[K:\mathbb{Q}]=m$, $\mathcal{O}_K$ be the ring of integers, and $N(J)$ denote the absolute norm of an ideal $J$. Let $I(n)$ denote the number of ideals of absolute norm $n$. Then the Dedekind zeta function of $K$ is given by
$$\zeta_K(s)=\sum_{I\subseteq \mathcal{O}_K}\frac{1}{N(I)^s}=\sum_{n=1}^{\infty}\frac{I(n)}{n^s},$$
in its domain of absolute convergence. If $\Delta_K$ denotes the discriminant of $K$, $r_1$ is the number of real embeddings and $r_2$ is the number of conjugate pairs of complex embeddings of $K$, so that $m=r_1+2r_2$, then $\zeta_K(s)$ satisfies the functional equation
\begin{align}
    \Lambda(s)=\Lambda(1-s),
\end{align}
where $\Lambda(s)=|\Delta_K|^{\frac s2}2^{-sr_2}\pi^{-\frac {sm}{2}}\ \Gamma^{r_1}\left(\frac s2\right)\Gamma^{r_2}(s)\zeta_K(s)$. Note that 
$$I(n)\leq d_m(n)\ll_{m,\varepsilon}n^{\varepsilon},$$
where $d_m(n)$ is the $m$-fold divisor function. 

The case $r_1=2$ and $r_2=0$ corresponds to $K$ being a real quadratic field. Upon taking $b(n)=I(n)$ and $\mu_n=\pi n/\sqrt{|\Delta_K|}$, the associated Dirichlet series satisfies the functional equation considered in Theorem \ref{truncatedformula2} with $p=0$. Notable examples of real quadratic fields include $\mathbb{Q}(\sqrt{2})$, $\mathbb{Q}(\sqrt{3})$, and $\mathbb{Q}(\sqrt{5})$.

However, when $r_1=0$ and $r_2=1$, that is, when $K$ is an imaginary quadratic field, we take $b(n)=I(n)$ and $\mu_n=2\pi n/\sqrt{\Delta_K}$. The resulting Dirichlet series satisfies the functional equation considered in Theorem \ref{truncatedformula} with $r=1$. Examples of such fields include $\mathbb{Q}(i)$, $\mathbb{Q}(\sqrt{-3})$, and $\mathbb{Q}(\sqrt{-5})$.

In either case, by Theorem \ref{th1}, we have that $x^{-\frac 14}D_{0}(x)$ is $\mathcal{B}^4$-almost periodic.

\section{Acknowledgement}
The second author's research was funded by the ANRF under File No. MTR/2023/000837 and CRG/2023/002698.

\end{document}